\documentclass[11pt,reqno]{article}
\usepackage{amsmath,amssymb,amsthm}
\usepackage{longtable}
\usepackage{xcolor}
\usepackage[margin=1in]{geometry}
\usepackage{adjustbox}
\usepackage{hyperref}
\usepackage{tikz}
\usepackage{tikzpingus}
\usepackage[numbers]{natbib}  % numeric citations like [3]

\theoremstyle{theorem}
\newtheorem{theorem}{Theorem}
\newtheorem{lemma}[theorem]{Lemma}
\newtheorem{proposition}[theorem]{Proposition}
\newtheorem*{theorem*}{Theorem~2}
\newtheorem{conjecture}[theorem]{Conjecture}

\theoremstyle{definition}
\newtheorem{defi}[theorem]{Definition}

\newtheorem{rem}[theorem]{Remark}
\newtheorem{problem}{Problem}
\newtheorem{question}{Question}

\usepackage{tikz}
\usetikzlibrary{arrows.meta,positioning,calc,decorations.pathreplacing}

\newcommand{\FvF}{{\rm FvF}}
\newcommand{\FvA}{{\rm FvA}}
\newcommand{\AvF}{{\rm AvF}}
\newcommand{\AvA}{{\rm AvA}}
\newcommand{\ava}{{\scriptscriptstyle {\rm AvA}}}
\newcommand{\fvf}{{\scriptscriptstyle {\rm FvF}}}
\newcommand{\fva}{{\scriptscriptstyle {\rm FvA}}}
\newcommand{\avf}{{\scriptscriptstyle {\rm AvF}}}
\newcommand{\X}{\mathsf{d}(X)}
\newcommand{\dual}[1]{\mathsf{d}#1}

\usepackage{authblk}
\author[1]{Anjali Bhagat}
\author[2]{Tanmay Kulkarni}
\author[1]{Urban Larsson\thanks{ORCID: \href{https://orcid.org/0000-0003-3663-1720}{0000-0003-3663-1720}}}
\author[1]{Divya Murali\thanks{\textit{E-mail addresses:} anjali.bhagat@iitb.ac.in, tanmayhk@uw.edu, larsson@iitb.ac.in, divyamurali963@gmail.com }}

\affil[1]{Department of IEOR, Indian Institute of Technology Bombay}
\affil[2]{University of Washington}

\title{Tie-breaking in self interest cumulative subtraction games}

\begin{document}

\maketitle

\begin{abstract}
Subtraction games have a rich literature as normal-play combinatorial games (e.g., Berlekamp, Conway, and Guy, 1982). Recently, the theory has been extended to zero-sum scoring play (Cohensius et al.\ 2019). Here, we take the approach of cumulative self-interest games, as introduced in a recent framework preprint by Larsson, Meir, and Zick. By adapting standard Pure Subgame Perfect Equilibria (PSPE) from classical game theory, players must declare and commit to acting either ``friendly'' or ``antagonistic'' in case of indifference. 
Whenever the subtraction set has size two, we establish a tie-breaking rule monotonicity: a friendly player can never benefit by a {\em deterministic \em deviation} to antagonistic play. This type of terminology is new to both ``economic'' and ``combinatorial'' games, but it becomes essential in the self-interest cumulative setting. The main result is an immediate consequence of the tie-breaking rule's monotonicity; in the case of two-action subtraction sets, two antagonistic players are never better off than two friendly players, i.e., their PSPE utilities are never greater. For larger subtraction sets, we conjecture that the main result continues to hold, while tie-breaking monotonicity may fail, and we provide empirical evidence in support of both statements.
\end{abstract}

\section{Introduction}
% Introduction
We study two-player alternating play perfect information games and we adapt the following definition of a self-interest symmetric cumulative subtraction game $S\subset \mathbb N = \{1,2,\ldots \}$. There is a single heap of pebbles, of size $h\in \mathbb N_0$, and the players alternately remove $s\in S$ pebbles from the heap, until the current player is not able to move from some position $h$ because, for all $s\in S, h-s<0$. The players are called Alice and Bob, and whatever they remove from the common heap, they put in their own individual pockets. Each player aims to maximize the size of their pocket (i.e., their cumulative utility) when the game ends. 

We adapt standard tools from game theory, particularly the classical backward induction procedure, to player profiles in Pure Subgame Perfect Equilibrium (PSPE). To avoid ambiguity in the recursive evaluation and to ensure the game tree is well-defined, we assume {\em deterministic tie-breaking} in cases of indifference. We refer to a pair of tie-breaking conventions, one for each player, as deterministic if they satisfy six properties. They cannot be random or hidden in any way, and they must be \emph{universal}, meaning that they are applied to every position whenever there is indifference between options, with respect to PSPE-computation. Moreover, they must be \emph{uniform}, meaning that the same conventions are used across all positions. They must be \emph{Markovian} (or history-independent), that is, they cannot depend on previous actions in any way. Finally, they are  \emph{unilateral}, meaning that each player's convention is independent of the other player's choice. One can prove by induction that PSPE play is indeed well defined by using these tie-breaking axioms.

Using a pair of deterministic tie-breaking conventions, players still maximize their own utility. However, when multiple options yield the same personal outcome, the choice between them may influence the opponent’s utility. We classify this influence as \emph{friendly} when it benefits the opponent and \emph{antagonistic} when it harms them. This effect is understood to be \emph{local}, meaning that it depends only on outcome discrepancies within any given pair of deterministic tie-breaking conventions. As shown in the Appendix, page~\pageref{sec:unidev}, such local choices can have globally counterintuitive consequences when viewed across ranges of possible conventions. 

Deterministic tie-breaking conventions are commitments that are applied consistently throughout the game tree, rather than as decisions made at individual nodes. For instance, given a pair of tie-breaking conventions, for every heap size, even one as large as $10^{10}$, a player is required to know whether they are indifferent. If so, for example, a player who adopts a friendly convention must then select a move that maximizes the opponent’s result, based on the already computed PSPE utilities of the resulting subpositions.

Following standard assumptions in combinatorial game theory, we assume the existence of such perfectly rational and computationally capable players. When both players adopt the friendly convention, we refer to the resulting setting as FvF or the {\em friendly} convention. If both adopt antagonistic conventions, we denote it AvA or the {\em antagonistic} convention. These will serve as our primary self-interest frameworks throughout the paper.\footnote{In our model, the tie-breaking rule is fixed in advance and is not selected as part of the PSPE; analyzing strategic choice over tie-breaking behavior would require a different model, such as a simultaneous-play normal form with mixed strategies. Here, we focus on the recursive structure and properties of combinatorial games instead.}

Subtraction games have been studied as win/loss games, mostly in the last-move convention ``a player who cannot move loses'' \cite{BCG1982, LS2025}, but also more recently as zero-sum scoring-play games \cite{CLMW2019}. The current self-interest setting was introduced in the framework preprint \cite{LMZ}. Let us first recall that setting, and after that (in Definition~\ref{def:outcomes}) we will introduce a slight generalization that will be convenient for the purpose of this paper and will also raise new interesting research problems.

\subsection{The symmetric outcome function}
Consider an antagonistic self-interest cumulative subtraction game. Following \cite{LMZ}, let us define the self-interest {\em outcome function}, under symmetric tie-breaking, of a game $S\subset \mathbb N$. The outcome for the two players is the ordered pair $o_\ava(h) = (o^1_\ava(h), o^2_\ava(h))$, where $h\in \mathbb N_0=\mathbb N\setminus \{0\}$ represents the current heap size. 
The cumulative utilities (a.k.a individual outcomes) can be recursively calculated using the outcome of heaps of smaller size. Let the first player choose $s \in S$ as their first move in PSPE play. Then $o^1(h) = s + o^2(h-s)$ and $o^2(h) = o^1(h-s)$. Thus, after the first move, the heap size is reduced to $h-s$, and the order of the players is switched (since the next move is made by the originally second player). Here, whenever there is {\em indifference} (i.e., more than one action leading to the same PSPE cumulation for the current player), the action chosen minimizes the PSPE cumulation for the other player, representing antagonistic behavior. In the case of {\em friendly} cumulative play, the individual outcomes $o^1_\fvf$ and $o^2_\fvf$ are analogous, but the action chosen maximizes the other player's cumulation in case of indifference, representing friendly behavior. 

Note that these two settings remain {\em symmetric}, in the sense that the outcome pairs do not depend on ``who starts'', Alice or Bob, and therefore, unless otherwise stated, we use the convention that Alice starts.  Thus, the setting represents a distant relative of the classical `impartial' combinatorial games. The outcome definition gives us an algorithm to compute perfect player results/utilities.

Let the discrepancy of player $i$ be $\delta^i(h) = o^i_\ava(h)-o^i_\fvf(h)$, for $i\in \{1,2\}$. A valid question is whether this discrepancy will be non-zero.
%both players adapting friendly tie-breaking rules versus both players adapting antagonistic tie-breaking rules. 
In Table~\ref{tab:3_5} we list some initial outcomes and discrepancies for the two-action ruleset $S=\{3,5\}$. %for some initial heap sizes. 
\begin{table}[h!]\caption{The subtraction set is $S=\{3,5\}$. The first non-zero discrepancy appears at $x=14$, by the second player losing one point if both players choose antagonistic tie-breaking rules.}
  \label{tab:3_5}
  \centering
  \resizebox{\textwidth}{!}{  
\begin{tabular}{|c| c c c c c c c c c c c c c c c c|}
\hline
$x$		&0 &1 &2 &3 &4 &5 &6 &7  & 8&9 &10&11&12&13&14&15\\ 
\hline
$o_\fvf(x)$	&$(0,0)$ &$(0,0)$ &$(0,0)$ &$(3,0)$ &$(3,0)$ &$(5,0)$  &$(5,0)$ &$(5,0)$ & $(5,3)$& $(5,3)$&$(5,5)$&$(6,5)$&$(6,5)$ &$(8,5)$&$(8,6)$&$(10,5)$\\ 
\hline
$o_\ava(x)$	&$(0,0)$ &$(0,0)$ &$(0,0)$ &$(3,0)$ &$(3,0)$ &$(5,0)$  &$(5,0)$ &$(5,0)$ & $(5,3)$& $(5,3)$&$(5,5)$&$(6,5)$&$(6,5)$ &$(8,5)$&$(8,5)$&$(10,5)$\\ 
\hline
$\delta^1(x)$	& 0&0 &0 &0 &0 &0 &0 &0 &0 &0 &0&0&0 &0&0&0\\
\hline
$\delta^2(x)$ & 0&0 &0 &0 &0 &0 &0 &0 &0 &0 &0&0&0 &0&$\boldsymbol{-1}$&0\\
\hline
\end{tabular}}
\end{table}
At heap size $14$, Bob loses one unit in the antagonistic setting. The PSPE move sequence in the friendly setting is 3-3-5-3, where Alice `sacrifices' in her first removal. In the antagonistic setting, she instead plays `greedily', and the PSPE sequence becomes 5-5-3, where Alice gets both the first and the last move. We experimentally discovered a number of such examples of discrepancies in the two and three-action settings, summarized in Section~\ref{sec:FvFvAvA}, with related conjectures. 

The main result of this paper, Theorem~\ref{thm:main}, is that, if the subtraction set has size 2,  then, for any heap size, each player's utility is no worse in the FvF setting than in the AvA setting. We conjecture that this result continues to hold if $|S|>2$; see Section~\ref{sec:2}.

\begin{rem}
Readers familiar with zero-sum games may note that the move sequence discussed in the context of Table~\ref{tab:3_5}  cannot appear in an optimal-play zero-sum setting, where player removals add and subtract to a common score, respectively. Alice would instead start by removing 5, and hence Bob must take 5, and Alice gets the last move by taking 3. This, of course, leads to a loss in ``social welfare'' as the whole heap would not be consumed. 
 A related topic concerns similarities between the $\AvA$ antagonistic and the zero-sum settings. We explore discrepancies between these games in a follow-up paper \cite{Tanmay}.% where we prove that if $|S|=2$ (stated here as a conjecture in Section~\ref{sec:AvAvZs}), PSPE play coincides for zero-sum games and the antagonistic setting. However, experimental studies reveal that discrepancies appear already for $|S|=3$ (related conjectures appear in Section~\ref{sec:AvAvZs}).
\end{rem}
%Let us return to the main topic for this paper, concerning discrepancies between friendly and antagonistic tie-breaking rules. 

If the players play several rounds of the same game, and both have so far both been committed to the friendly setting, could either player benefit by instead deviating to deterministic antagonistic?

%This question induces a refined setting; .
\begingroup
\sffamily
%\texttt{
\begin{quote}
On a remote island, two clans of Penguins thrived. Over generations they had developed an unusual gift: a talent for recursive computation in games played with heaps of fish. In the picture, Alice appears as the friendly, fish-pile-wise Penguin, while Bob, equally skilled, deviates and commits to antagonistic play in cases of indifference. By natural selection, only the fittest fish-counters survive.\footnote{Images thanks to https://github.com/EagleoutIce/tikzpingus.}
\end{quote}
\par
\endgroup
\vspace{1 mm}
\begin{center}
\begin{tikzpicture}[scale=.9]

% Penguins
\pingudefaults{wings grab, eyes shiny} 
\pingu[left wing shock, eyes angry, xshift=2.1cm, yshift=1.3cm, eye patch left, glow solid=purple] 
\pingu[glow solid=yellow, wings wave, xshift=-2.8cm, yshift=1.3cm]

% Fish (Bottom row)
\foreach \x/\angle in {0/10, 0.25/-5, 0.5/15, 0.75/-10, 1/5} {
    \begin{scope}[shift={(\x*1.2, 0.1)}, rotate=\angle]
        % Body
        \filldraw[fill=white!60!blue, draw=black, line width=0.3pt] (0,0) ellipse (0.28 and 0.15);
        % Tail closer
        \filldraw[fill=white!60!blue, draw=black] (0.27,0) -- (0.34,0.07) -- (0.34,-0.07) -- cycle;
        % Eye
        \fill[black] (-0.22, 0.05) circle (0.02);
    \end{scope}
}

% Middle row
\foreach \x/\angle in {0.125/5, 0.375/-8, 0.625/12, 0.875/-6} {
    \begin{scope}[shift={(\x*1.2, 0.4)}, rotate=\angle]
        \filldraw[fill=white!60!blue, draw=black, line width=0.3pt] (0,0) ellipse (0.28 and 0.15);
        \filldraw[fill=white!60!blue, draw=black] (0.27,0) -- (0.34,0.07) -- (0.34,-0.07) -- cycle;
        \fill[black] (-0.22, 0.05) circle (0.02);
    \end{scope}
}

% Top row
\foreach \x/\angle in {0.25/7, 0.75/-5} {
    \begin{scope}[shift={(\x*.9, 0.7)}, rotate=\angle]
        \filldraw[fill=white!60!blue, draw=black, line width=0.3pt] (0,0) ellipse (0.28 and 0.15);
        \filldraw[fill=white!60!blue, draw=black] (0.27,0) -- (0.34,0.07) -- (0.34,-0.07) -- cycle;
        \fill[black] (-0.22, 0.05) circle (0.02);
    \end{scope}
}

% Peak
\begin{scope}[shift={(.5, 0.9)}]
    \filldraw[fill=white!60!blue, draw=black, line width=0.3pt] (0,0) ellipse (0.28 and 0.15);
    \filldraw[fill=white!60!blue, draw=black] (0.27,0) -- (0.34,0.07) -- (0.34,-0.07) -- cycle;
    \fill[black] (-0.22, 0.05) circle (0.02);
\end{scope}

%Bob's cumulation
\foreach \x/\angle in {2/10, 2.25/-5} {
    \begin{scope}[shift={(\x*2.2, 0.1)}, rotate=\angle]
        % Body
        \filldraw[fill=white!60!blue, draw=black, line width=0.3pt] (0,0) ellipse (0.28 and 0.15);
        % Tail closer
        \filldraw[fill=white!60!blue, draw=black] (0.27,0) -- (0.34,0.07) -- (0.34,-0.07) -- cycle;
        % Eye
        \fill[black] (-0.22, 0.05) circle (0.02);
    \end{scope}
}
%Alice's cumulation
\foreach \x/\angle in {-1.55/-15} {
    \begin{scope}[shift={(\x*2.2, 0.1)}, rotate=\angle]
        % Body
        \filldraw[fill=white!60!blue, draw=black, line width=0.3pt] (0,0) ellipse (0.28 and 0.15);
        % Tail closer
        \filldraw[fill=white!60!blue, draw=black] (0.27,0) -- (0.34,0.07) -- (0.34,-0.07) -- cycle;
        % Eye
        \fill[black] (-0.22, 0.05) circle (0.02);
    \end{scope}
}

\end{tikzpicture}
\end{center}
\vspace{3mm}
The Penguin narrative will be revisited in the Appendix. 
\begin{question}\label{ques:1}
Could any personalized player benefit from a unilateral deviation from a deterministic friendly tie-breaking?
\end{question}
This question requires a generalized outcome function, and we devote Section~\ref{sec:2} to this.

\subsection{Outline}
Let us outline the content of this paper. In Section~\ref{sec:2} we generalize the outcome function to four deterministic tie-breaking pairs and explain the road map to the main theorem for two-action rulesets. In Section~\ref{sec:main}, we prove the main theorem, Theorem~\ref{thm:main}, via a few lemmas, highlighting first player advantage, greedy play vs. sacrifice, tie-breaking monotonicity and more. In Section~\ref{sec:FvFvAvA}, we outline a follow-up project with several structurally detailed conjectures concerning FvF vs AvA (those are the conjectures that motivated this study), while in Section~\ref{sec:AvAvZs}, we outline some conjectures by instead comparing AvA with zero-sum play. In Section~\ref{sec:future}, we discuss further open problems, and in Section~\ref{sec:related}, we mention some related work. Finally, in the Appendix (page \pageref{sec:unidev}), we illustrate how cases where $|S|>2$ can differ from two action rulesets, in globally counter intuitive ways. 

\section{A generalized outcome and the main result}\label{sec:2}
In combinatorial games, we define perfect play outcome functions as backward induction algorithms to compute PSPE; in asymmetric cases, such as ``partizan play'', perfect play result can be sensitive to who starts. Observe that in our setting, Question~\ref{ques:1} induces two more settings, and a more general outcome function. Let us list all settings.  
\begin{itemize}
\item[\bf{FvF:}] both players act friendly, in case of indifference;
\item[\bf{FvA:}] Alice, as a starting player, acts friendly, and Bob is antagonistic, in case of indifference;
\item[\bf{AvF:}] Bob, as the second player, acts friendly, whereas Alice is antagonistic, in case of indifference;
\item[\bf{AvA:}] both players are antagonistic, in case of indifference.
\end{itemize}
This includes assymetric settings, by their tie-breaking commitments, although the ruleset itself is symmetric. However, by defining the notion of a dual setting, we can still use a relatively compact definition of ``outcome''. 
%Observe that this requires a generalization of the outcome function. 

Let \( \tau = \{\FvF, \AvF, \FvA, \AvA\} \), and let $X\in \tau$. The \emph{dual} of \( X \), denoted \( \X \), is defined by:
\[
\X =
\begin{cases}
\FvF & \text{if } X = \FvF, \\
\FvA & \text{if } X = \AvF, \\
\AvF & \text{if } X = \FvA, \\
\AvA & \text{if } X = \AvA.
\end{cases}
\]
Thus, for all $X$, $\dual(\X)=X$ and $\X=X$ if and only if $X\in \{\FvF, \AvA\}$. 
\begin{defi}[The Outcome Function]\label{def:outcomes}
Consider a heap size $h\ge 0$ and $S(h):=S\cap \{1,\ldots ,h\}$. Let, for all \( X \in \tau \),  $o^1_X(h)=o^2_X(h)= 0$, if $h<\min S$. Otherwise, let
\begin{align*}
o^1_X(h) &= \max\{o^2_{\X}(h - s) + s \ | \ s \in S(h)\},\\
o^2_X(h) &= o^1_{\X}(h - s^*),
\end{align*}
where $$s^* = \textrm{argmin}_{s \in S^*(h)} \ o^1_{\X}(h - s),$$ if $X\in \{\AvF, \AvA\}$, and otherwise 
$$s^* = \textrm{argmax}_{s \in S^*(h)} \ o^1_{\X}(h - s).$$ 
Here $S^*(h) \subseteq S$ denotes the set of {\em indifferent actions}; this set contains the set of actions the current player can take from $h$ that give them the same cumulation.\footnote{There is some abuse of notation, since this action may not be unique; but if it is not unique, then there is no meaningful difference between them, as they will lead to the same PSPE cumulations for both players, when the game ends.}
For all $X\in \tau$, for all $h\ge 0$, let $o_X(h)=(o^1_X(h),o^2_X(h))$ be the $X$-outcome of $h$.
\end{defi}

Given a tiebreaking convention, we usually refer to a pair of (EGT) utilities in PSPE simply as a (CGT) {\em outcome}. As mentioned, even if the ruleset is symmetric, since the players may have different tie-breaking rules, a game played is not necessarily symmetric in the strict sense. If one player is antagonistic at some stage of play, then they are committed to being antagonistic at all stages of play, and similarly for ``friendly''. This is a situation that cannot appear in standard CGT impartial games. In this sense, the move options depend on their committed tie-breaking rule. To the best of our knowledge, such situations have not yet been studied in the literature.

Thus, there are four deviation cases to consider, listed as  FvF $\rightarrow$ AvF, FvF $\rightarrow$ FvA, AvA $\rightarrow$ FvA and AvA $\rightarrow$ AvF. For example, in FvF $\rightarrow$ AvF, in the FvF setting, both players resolve ties in a friendly manner, while in AvF, Alice (playing first) deviates with an antagonistic tie-breaking rule. 

 In the case \( |S| = 2 \), we here prove that a unilateral deviation from friendly tie-breaking to antagonistic ditto never benefits either player; as we will see, whenever starting from FvF, the deviating player will be unaffected, while the other player may suffer in some cases (Lemma~\ref{lem:fvfdeviate}). As a consequence, we will get:

\begin{theorem}[Main Theorem]\label{thm:main}
Consider any subtraction set $S$, with $|S|=2$, and suppose that both players act friendly in case of indifference. Then each player's PSPE utility is never worse than if both players have antagonistic tie-breaking rules. 
\end{theorem}

 Our methods to prove this result do not apply in general if \( |S| > 2 \),  and we provide a counter example, when unilateral deviation is non-monotone, namely if $S=\{4,5,9\}$, in Table~\ref{tab:deviation459} in the Appendix. 
We will study the two unilateral deviations from FvF, to FvA and AvF, while the converse/symmetric deviations from AvA, to AvF and FvA are analogous. A unilateral deviation from friendly to antagonistic may, in rare cases, be beneficial to the deviating player; still, we have numerical evidence that Theorem~\ref{thm:main} continues to hold for any subtraction set of size greater than 2 (see Section~\ref{sec:future}). However, we currently have no formal method to jump directly from both friendly to both antagonistic.

 \begin{conjecture}\label{con:main}
Consider any subtraction set $S$, and suppose both players act friendly in case of indifference. Then each player's PSPE utility is never worse than if both players have antagonistic tie-breaking rules. 
\end{conjecture}

%\ur{This seems misplaced.. it jumped from somewhere?} 

 %We call any move in PSPE by {\em optimal}, with respect to the current player. 
  %We begin by studying the cases where exactly one of the players deviate from both beeing friendly. Later, in Section~\ref{sec:converse}, we will study the converse situation, where exactly one of the players deviates to become friendly when both were antagonistic. 
 Consider ordered pairs of integers $(a,b)$ and $(c,d)$. Then $(a,b)\le (c,d)$ if $a\le c$ and $b\le d$, and similarly with ``$\ge$''. 

\section{Unilateral tie-breaking deviation and a proof of the main result}\label{sec:main} 

Suppose that both players follow deterministic tie-breaking conventions. We are interested in comparing the outcome under one such setting, where both players are universally friendly, with the setting in which exactly one of the players deviates, by committing to an antagonistic tie-breaking rule.

Whenever the subtraction set, $S=\{s_2,s_1\}$,  has exactly two actions, then we claim that the deviating player's PSPE-utilities will be exactly the same as if they had remained friendly, and the other player's PSPE-utilities cannot improve. This claim holds independently of who starts, and we prove it in Lemma~\ref{lem:fvfdeviate} in this section.  

In the following sections, we will fix the convention $s_2<s_1$. 

Through Lemma~\ref{lem:fvfdeviate}, we see how the claims in the first paragraph combine to form an induction proof.  For any two-action subtraction set $S$, let $\Delta=s_1-s_2>0$. A {\em greedy} move from $h\ge s_1$ is to play to $h-s_1$, and otherwise a player {\em sacrifices}. The {\em endgame} refers to the play to a heap size $h<s_1$. The {\em slack} is a leftover that cannot be attained by either player, and therefore refers to a heap of size $h<s_2$. 

There is a general proof of the next lemma in the zero-sum setting \cite{CLMW2019}. Here, it only holds in the two-action setting. See Table~\ref{tab:deviation459} in the Appendix on page~\pageref{sec:unidev}, where $S=\{4,5,9\}$ and where heaps of sizes $61,75,\ldots$ have strictly better PSPE utility for the second player. 

\begin{lemma}[First Player Advantage]\label{lem:start}
    If $|S|=2$, then the starting player's utility in PSPE is never worse than that of the second player.
\end{lemma}
\begin{proof}
Recall $s_1>s_2$. Suppose that Alice plays greedily until she cannot play $s_1$. Then she gets utility $ks_1+ \ell s_2$, some $k>0$, some $\ell\ge 0$. If Bob gets to play $k$ copies of $s_1$, then Alice has the first player advantage at the endgame too, and hence Bob cannot get more than $ks_1+ \ell s_2$. If Bob can only collect $k-1$ copies of $s_1$, then he can at most get one more copy of $s_2$  than Alice, but their utility difference is at least $\Delta>0$ in Alice's favor. 
\end{proof}

We distinguish two parameter regimes for $S=\{s_2,s_1\}$:
\begin{itemize}
\item {\em dominance}: $2s_2 \le s_1$; 
\item {\em non-dominance}: $2s_2 > s_1$. 
\end{itemize}
In case of dominance, the boundary case $2s_2 = s_1$ is the {\em balanced regime}.

The usefulness of a sacrifice may appear limited while playing on a single heap of tokens. However, simple arithmetic shows that if $S=\{2,3\}$ then Alice benefits strictly when playing from a heap of size $7$, and if  $S=\{3,4\}$, then she can benefit strictly by sacrificing twice, playing from $h=17$. There are many more such examples, which are analogous to the zero-sum setting studied in \cite{CLMW2019}. In two-action zero-sum games, the usefulness of a sacrifice is limited to basic arithmetic considerations for small heap sizes. 
However, self-interest play appears to be more multifaceted. As shown in Table~\ref{tab:3_5}, there are situations where, under a deterministic tie-breaking rule, Alice sacrifices solely to benefit Bob, and Bob's immediate response is also a sacrifice. Note that the example occurs in the non-dominant regime. By contrast, the dominant regime generally exhibits less complex behaviour.

\begin{lemma}[Dominance and Greedy]\label{lem:sacr}
Consider PSPE play under any deterministic tie-breaking rule, and let $S=\{s_2,s_1\}$ with $s_2<s_1$. In the dominant regime, i.e. if 
\begin{equation}\label{eq:obs}
2s_2\le s_1,
\end{equation}
then a player cannot benefit strictly by sacrificing. 
\end{lemma}
\begin{proof}
The proof will use a strategy-stealing argument to establish a contradiction. Suppose that Alice sacrifices in her first move and profits strictly in PSPE play. This means that she must get at least $k$ more $s_2$ actions to compensate for her initial sacrifice, if $s_1 = ks_2 + r$, $0\le r< k$. However, since $k\ge 2$, this requires that Bob also plays a sacrifice at some point. Otherwise, Bob gets at least $3\Delta-s_1\ge \Delta>0$ larger utility than Alice, while Lemma~\ref{lem:start} assures that Alice, as a starting player, never gets worse than Bob in PSPE play. But Bob can replace this sacrifice with $s_1$, by stealing $ks_2 + r\le (k+1)s_2$ of Alice's required removal segment. Hence, Alice's first sacrifice cannot lead to a strict gain.
\end{proof}

A consequence of this result is that, in the case of dominance, $s_1$ is never a bad choice for the current player (if feasible); see Proposition~\ref{prop:equality}.  %Indeed, experimental results indicate consistently zero discrepancies in the dominant regime.  

In a zero-sum combinatorial game setting, parity consideration can be central, determining whether you or the opponent gets the last move. Here, it remains a decisive factor in the following sense. The concept of a `last move' is individualized; for all sufficiently large heaps (depending on $S$) both players have a last move. Sometimes you can sacrifice and secure some parity advantage; you win a last move that you otherwise would not have reached.

We are often interested in understanding the play from the options of a heap of size $h$. This motivates us to denote the branches $H_1=h-s_1$ and $H_2=h-s_2$, and think of the $H_i$s as games with full game trees, rather than `heap sizes'.

\begin{lemma}[Last Move and Sacrifice]\label{lem:lastmove}
If a player sacrifices playing in PSPE from a heap of size $h$, independently of tie-breaking rules, they know that one of the two scenarios applies:
\begin{itemize}
    \item[(i)] they will be able to replace their last move $s_2$ in $H_1$ for $s_1$ in the $H_2$ branch.
    \item[(ii)] they will be able to play a last move on the $H_2$ branch that they would not have been able to play on the greedy branch $H_1$ (they get one additional move on the $H_2$ branch), and the regime is not strictly dominant;
\end{itemize}
\end{lemma}
\begin{proof}
Suppose that they sacrifice without either of the items being satisfied. Then play greedily on $H_1$, and get a larger utility, by at least $\Delta$ on this branch. 

More explicitly, consider (i). Here, their last move $s_2$ in $H_1$ is replaced by $s_1$ in $H_2$, giving them a gain of $\Delta$ in $H_2$. Meanwhile, the first move in $H_1$ and $H_2$ is $s_1$ and $s_2$ respectively. Therefore, in the first move, they have a loss of $\Delta$ in $H_2$. Thus, with the gain and loss balanced out in $H_1$ and $H_2$, both sacrifice and greedy play are in PSPE.

In case of (ii), the sacrificing player loses $\Delta = s_1-s_2$ in their first move. Since the slack, at the end of $H_1$, is strictly less than $s_2$, the gain in the added last move in $H_2$ is strictly less than $s_2 + \Delta = s_1$. Thus, the last move can only be $s_2$. Therefore, in the strictly dominant regime ($\Delta > s_2$), the $H_2$ branch gives a strictly lesser utility, and thus, sacrifice cannot be in PSPE. In a balanced regime ($\Delta = s_2$), a player performs equally in $H_1$ and $H_2$, and in a non-dominant regime ($\Delta <s_2$), the $H_2$ branch gives a strictly better utility. Thus, in both these cases, sacrifice play is in PSPE.  

Now, suppose a player sacrifices, but neither (i) nor (ii) applies. Then on the branch $H_2$ they cannot obtain an extra move and cannot upgrade a final $s_2$ to $s_1$. Thus, their subsequent utility from $H_2$ is the same as from $H_1$.  But they are behind by $\Delta=s_1-s_2$ already after the first move. Hence, the greedy branch $H_1$ gives a strictly larger utility. This contradicts the assumption that the sacrifice is played in PSPE. Therefore, in PSPE any sacrifice must be justified by either (i) or (ii).  
\end{proof}

\begin{lemma}[Tie-breaking Monotonicity, FvF]\label{lem:fvfdeviate}
    Consider $S=\{s_2,s_1\}$, with $s_2<s_1$. Then, for any heap size $h$, 
    \begin{enumerate}
        \item $o^1_{\fvf}(h)\ge o^1_{\fva}(h)$; 
        \item $o^2_{\fvf}(h)=o^2_{\fva}(h)$; 
        \item $o^1_{\fvf}(h) = o^1_{\avf}(h)$; 
        \item $o^2_{\fvf}(h) \ge o^2_{\avf}(h)$. 
    \end{enumerate}
\end{lemma}
\begin{proof}
    The statement holds for $h<s_2$. Suppose the result holds for all heap sizes smaller than $h\ge s_2$. 
     Alice has at least one PSPE move in FvF; denote this by $s'$, and similarly, for a given $X\in \{\FvA,\AvF\}$, she has at least one PSPE move; denote this by $s''$. (Here $s'$ and $s''$ may differ or be the same.) Thus, we have:
    \begin{itemize}
        \item $o^1_{\fvf}(h)= s'+o^2_{\fvf}(h-s')$; 
        \item $o^2_{\fvf}(h)=o^1_{\fvf}(h-s')$; 
        \item $o^1_{X}(h)=s''+o^2_{\X}(h-s'')$; 
        \item $o^2_{X}(h)=o^1_{\X}(h-s'')$. 
    \end{itemize}
    
    Hence, with regard to the first two items in the statement (so $X=\FvA$), it suffices to prove that 
    \begin{enumerate}
        \item[(i)] $s'+o^2_{\fvf}(h-s')\ge s''+o^2_{\avf}(h-s'')$; 
        \item[(ii)] $o^1_{\fvf}(h-s')=o^1_{\avf}(h-s'')$.
    \end{enumerate}
    By induction, using item (4) in the statement, we have that $o^2_{\fvf}(h-s'')\ge o^2_{\avf}(h-s'')$, and thus, by the PSPE assumption on $s'$ in FvF, the inequality in item (i) holds.\\

    \noindent To prove the equality in item (ii), observe that if $s'=s''$, then the equality holds by induction, since, by item (3) in the statement, we may replace $o^1_{\fvf}(h-s')$ with $o^1_{\avf}(h-s')$. We study two cases when the actions differ.\\
    
    \noindent Case 1, $s'>s''$: Suppose that $s'=s_1$ and $s''=s_2$, corresponding to Alice's choices, playing from $h$ in FvF and FvA, respectively, and let $H_1=h-s_1$ and $H_2=h-s_2$. By induction, we may identify $o^1_{\fvf}(H_1)=o^1_{\avf}(H_1)$, and study both options in the same convention, say AvF. Thus, it suffices to prove that 
    \begin{align}\label{eq:avf_avf}
        o^1_\avf(H_1)=o^1_{\avf}(H_2),
    \end{align} 
    given that Alice sacrificed in her first PSPE move from $h$ in FvA. Recall, Bob plays first from the options of $h$. 
    
    First, we prove that Bob cannot do strictly better on the smaller heap $H_1=H_2-\Delta$ in PSPE. This follows because he can mimic (couple) all moves from the smaller heap to the larger heap. If Alice at some stage plays $s_1$ on $H_2$ and $s_2$ on $H_1$, then this coupled move pair compensates exactly the initial $\Delta$ difference, so the heaps are equalized. And, given PSPE play by Alice on both heaps, since they follow the same tie-breaking convention  \eqref{eq:avf_avf}, Bob's mimic strategy obviously remains a valid strategy. (If Alice plays a smaller move on $H_2$, but the larger one on $H_1$, Bob's strategy remains feasible.)
    
Now we show that Bob cannot do strictly better on the larger heap $H_2$.  Alice sacrificed in her first PSPE move on $H_2$, so we know that she does not have a parity disadvantage on that heap. Let us argue first in the case of dominance. By Lemma~\ref{lem:sacr}, Lemma~\ref{lem:lastmove} (ii) applies, and so the initial $\Delta$ loss will be compensated by $\Delta$ in the endgame. In case of non-dominance,  Lemma~\ref{lem:lastmove} (i) might apply. But in this case $\Delta < s_2$, so Bob's moves on the larger heap can be coupled on the smaller heap, because Alice's share of the larger heap increased by $s_2$ due to parity advantage.\\
 
\noindent Case 2, $s'<s''$: Suppose next that $s' = s_2$ and $s'' = s_1$. By the induction argument in Case 1, again it suffices to justify  \eqref{eq:avf_avf}, which is already established.\\

\noindent Regarding the last two items in the statement (so $X=\AvF$), it suffices to prove that 
\begin{enumerate}
    \item[(iii)] $s'+o^2_{\fvf}(h-s') = s''+o^2_{\fva}(h-s'')$; 
    \item[(iv)] $o^1_{\fvf}(h-s')\ge o^1_{\fva}(h-s'')$. 
\end{enumerate}
For item (iii), by induction, we have that 
\begin{align}\label{eq:o2s''}
    o^2_{\fvf}(h-s'')=o^2_{\fva}(h-s'')
\end{align}
and 
\begin{align}\label{eq:o2s'}
    o^2_{\fvf}(h-s')=o^2_{\fva}(h-s'). 
\end{align}

Hence item~(iii) holds, whenever Alice is indifferent between her move options from $h$ in FvF or AvF, or if $s'=s''$. Thus, the only case we must rule out is whenever the two settings have differing single-move  PSPE options. But then $s'+o^2_{\fvf}(h-s') > s''+o^2_{\fvf}(h-s'')$ and $s'+o^2_{\fva}(h-s') < s''+o^2_{\fva}(h-s'')$, where $s'$ and $s''$ are PSPE in FvF and FvA, respectively. Thus, by induction, using \eqref{eq:o2s''} and \eqref{eq:o2s'}, $s'+o^2_{\fvf}(h-s') > s''+o^2_{\fvf}(h-s'')$ and $s'+o^2_{\fvf}(h-s') < s''+o^2_{\fvf}(h-s'')$, a contradiction. 

For item (iv), the proof is analogous to the second part of (ii), but using inequality instead in \eqref{eq:avf_avf} via the induction step.

\end{proof}

We have an analogous lemma concerning unilateral friendly deviation from the AvA setting. 

\begin{lemma}[Tie-breaking Monotonicity,  AvA]\label{lem:avadeviate}
    Consider $S=\{s_2,s_1\}$, with $s_2<s_1$. Then, for any heap size $h$, 
    \begin{enumerate}
        \item $o^1_{\ava}(h)= o^1_{\fva}(h)$; 
        \item $o^2_{\ava}(h)\le o^2_{\fva}(h)$; 
        \item $o^1_{\ava}(h) \le o^1_{\avf}(h)$; 
        \item $o^2_{\ava}(h) = o^2_{\avf}(h)$. 
    \end{enumerate}
\end{lemma}
\begin{proof}
This is analogous to the proof of Lemma~\ref{lem:fvfdeviate}.
\end{proof}
Recall that the main result Theorem~\ref{thm:main} concerns subtraction sets of size two, and we have to prove that, if both players act friendly in case of indifference, then the outcome is never worse than if both players have antagonistic tie-breaking rules. 

\begin{proof}[Proof of Theorem~\ref{thm:main}]
Combine Lemmas~\ref{lem:fvfdeviate} and~\ref{lem:avadeviate}.
\end{proof}

In the case of dominance, we arrive at a stronger version of the main theorem.

\begin{proposition}\label{prop:equality}
    Consider $S = \{s_2, s_1\}$, in the dominant regime when $2s_2 \le  s_1$. Then there is no heap $h$ for which the $\FvF$ and $\AvA$ settings have a discrepancy. In more generality, for any heap size $h$, and any $X,Y \in \tau$, $o_{X}(h)= o_{Y}(h)$, i.e., 
    \begin{itemize}
        \item $o^1_{X}(h)= o^1_{Y}(h)$; 
        \item $o^2_{X}(h) = o^2_{Y}(h)$. 
    \end{itemize}
\end{proposition}
\begin{proof}
    From Lemma \ref{lem:sacr}, we know that in the dominant regime, a player cannot strictly benefit by sacrificing. We claim that greedy play is in PSPE independently of tie-breaking rules.

    Consider the $H_2$ branch, where the previous player (Alice) has made a sacrifice. Let there be strict dominance. Thus, we may assume that we are in case Lemma~\ref{lem:lastmove} (i), which implies that Alice lost $\Delta$ in her first move, and regained $\Delta$ in her last move. Thus, Bob's PSPE utility remains the same as in the $H_1$ branch, since he does not gain anything from Alice's sacrifice, irrespective of the tie-breaking rules. 
 
 In case of the balanced regime, both  Lemma~\ref{lem:lastmove} (i) and (ii) are possible, but in either case, the argument is the same as in the previous paragraph, since there is no change in Alice's utility in both cases. 
\end{proof}

\section{Friendly vs. antagonistic discrepancies}\label{sec:FvFvAvA}
We are interested in comparing outcomes between the $\FvF$ and $\AvA$ settings, specifically when the subtraction set $S$ has exactly two actions. In this scenario, the first such discrepancy between the outcomes appears when $S = \{3, 5\}$ and $h = 14$, as demonstrated in Table~\ref{tab:3_5}.

Further computer simulations of this scenario yield a list of two-action subtraction sets $S=\{s_2,s_1\}$ for which there exists a heap where the $\FvF$ and $\AvA$ outcomes differ. Our observations are summarized in Figure~\ref{fig:2-action_friendly_antagonistic}.

\begin{figure}[h!] 
    \centering
    \includegraphics[width=0.5\linewidth]{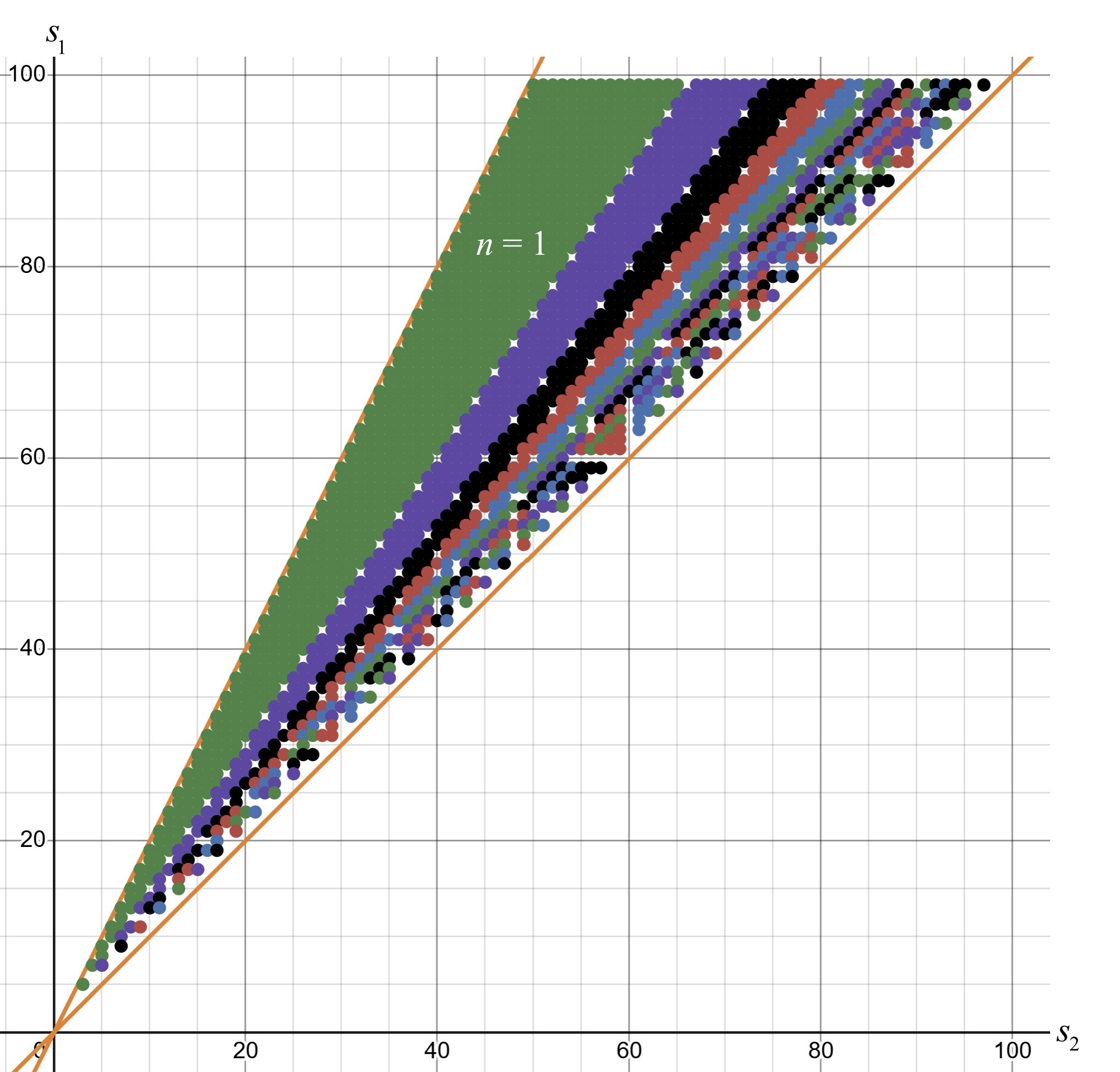}
    \caption{Rulesets $S = \{s_2, s_1\}$ (visualized as points $(s_2, s_1)$, with $s_2 < s_1$) for which there exists a heap with a antagonistic-friendly discrepancy. The colors indicate the value of $n$ in (\ref{eq:smallestheap}) modulo 5  for the given subtraction set (e.g., the large green area indicates $n = 1$, purple is $n = 2$, and so on up to $n = 48$). An online version can be found \href{https://www.desmos.com/calculator/erdndez4ey}{here}.}\label{fig:2-action_friendly_antagonistic}
\end{figure}

We note some distinct features within Figure~\ref{fig:2-action_friendly_antagonistic}: 
\begin{enumerate}
    \item[(1)]  no discrepancies occur in the region above $s_1 = 2s_2$,
    \item[(2)] a number of discrepancies occur with subtraction sets in the region below $s_1 = 2s_2$, and 
    \item[(3)] the subtraction sets with no discrepancy in the region below $s_1 = 2s_2$, occure on half-lines radiating from the origin. 
\end{enumerate}
%(1) no discrepancies occur in the region above $s_1 = 2s_2$, and (2) a number of discrepancies occur with subtraction sets in the region below $s_1 = 2s_2$, (3) the subtraction sets with no discrepancy in the region below $s_1 = 2s_2$, occure on half-lines radiating from the origin. 

The first observation is a direct consequence of Proposition \ref{prop:equality}, while the latter two observations underpin the following conjectures. 

\begin{conjecture} Consider $S = \{s_2, s_1\}$ where $\frac{s_1}{s_2} = \frac{k + 1}{k}$, for some $k \in \mathbb{N}$. Then there is no heap $h$ that distinguishes the outcomes of the two conventions $\FvF$ and $\AvA$.% games have a discrepancy. 
\end{conjecture}

\begin{conjecture}\label{con:halfline}  
Consider $S = \{s_2, s_1\}$ in the non-dominant regime where $s_1 < 2s_2$, and supppose, for all $k \in \mathbb{N}$, $\frac{s_1}{s_2} \ne \frac{k + 1}{k}$. Then there exists a heap $h$ that distinguishes the outcomes of the $\FvF$ and $\AvA$ conventions. The smallest such heap is defined by:
\begin{align}\label{eq:smallestheap}
h = (n + 2)s_2 +ns_1,
\end{align}
where $n \in \mathbb{N}$ satisfies $\frac{n + 1}{n} > \frac{s_1}{s_2} > \frac{n + 2}{n + 1}$.
\end{conjecture}
Note that, in Conjecture~\ref{con:halfline}, the existence and uniqueness of $n$ is clear, namely $n=\left\lfloor \frac{s_2}{s_1-s_2}\right\rfloor$. 
When we extend the methodology used in Figure~\ref{fig:2-action_friendly_antagonistic} to three-action games $S = \{s_3, s_2, s_1\}$, we use the convention $s_3 < s_2 < s_1$, and a number of new trends emerge; see Figure~\ref{fig:3-action_friendly_antagonistic}.

\begin{figure}[h!] 
    \centering
    \includegraphics[width=0.5\linewidth]{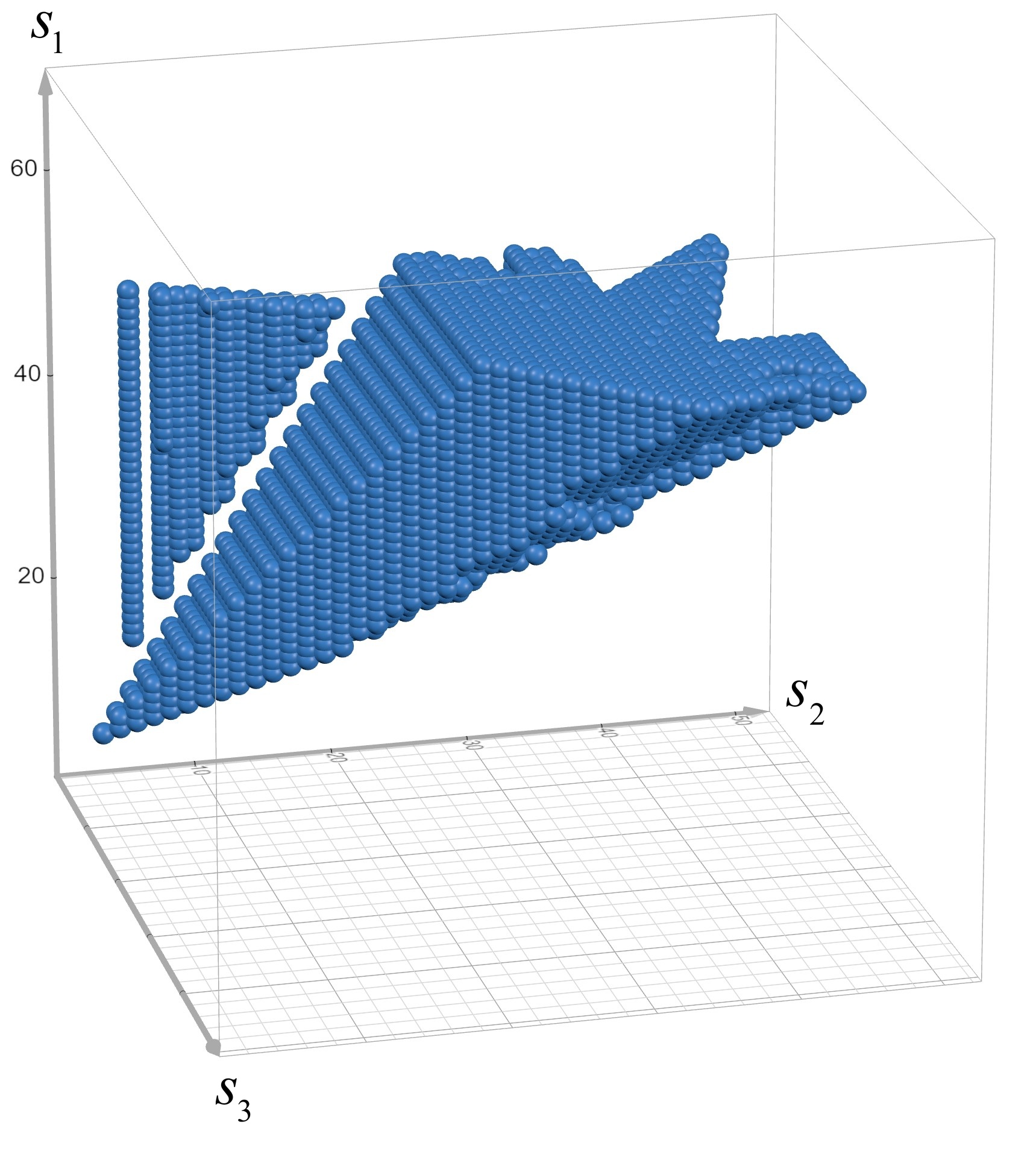}
    \caption{Points $(s_3, s_2, s_1)$ where there exists a heap with a antagonistic-friendly discrepancy for $S = \{s_3, s_2, s_1\}$, with $s_3<s_2<s_1$. An online rotatable version can be found \href{https://www.desmos.com/3d/38511dd74c}{here}.}\label{fig:3-action_friendly_antagonistic}
\end{figure}
\begin{conjecture} Consider $S = \{s_3, s_2, s_1\}$, with $s_1 \ge 2s_2$, and let $S' = \{s_3, s_2\}$. If there is no discrepancy between $\FvF$ and $\AvA$ in $S'$, then there is no discrepancy in $S$.
\end{conjecture}
Let us study the so-called {\em additive} rulesets. 
\begin{conjecture}
Consider $S = \{s_3, s_2, s_1\}$ where $s_1 = s_3 + s_2$. Then, there exists a heap $h$ with a discrepancy, if and only if $\frac{s_2}{s_3} \ne k$, for any $k \in \mathbb{N}$. The smallest such heap is:
$$h = (n + 2)s_2 + ns_1 = ns_3 + (2n + 2)s_2,$$
for $n \in \mathbb{N}$ satisfying $n + 1 > \frac{s_2}{s_3} > n$.
\end{conjecture}
%%%%%%%%%%%
\section{Antagonistic vs. zero-sum discrepancies}\label{sec:AvAvZs}
Let us consider zero-sum scoring play subtraction games. In this setting, Alice's cumulations add to a common score, and Bob's cumulations subtract from it. The outcome function \cite{CLMW2019} in this setting is:
$$o_{\textrm{zs}}(x) = 
\begin{cases}
    \max_{s \in S}\{s - o_{\textrm{zs}}(x - s)\}, & \text{if $x \ge$ min $S$} \\
    0, & \text{otherwise.}
\end{cases}$$
Recall that in antagonistic tie-breaking, both players minimize their opponent's utility whenever there is an indifference. This may appear similar to the zero-sum case where the players are always working against each other.
%In both zero-sum scoring play and antagonistic self-interest play, the game is built on a shared heap, so at a structural level, one player’s gain reduces the other’s potential. The difference is that in a zero-sum game, this opposition is explicit; Alice’s utility is defined as Bob’s loss. In $\AvA$, by contrast, opposition arises only at tie-breaking: outside of indifference, each player maximizes their own tally without regard to the other. 
In this spirit we searched for games where the $\AvA$ and scoring play settings diverge, but in the case of two-action games, we did not find any discrepancies. 

%Consider the $\AvA$ self-interest game and the zero-sum scoring play game for $S = \{s_2, s_1\}$, where $s_2 < s_1$. However, the graph produced  (using the same methodology as 
%Figure~\ref{fig:2-action_friendly_antagonistic}) 
%was completely empty, which we posit as a larger trend.

\begin{conjecture} 
Consider any $S=\{s_2,s_1\}$. Then, for all heap sizes $h$, $o_{\textrm{zs}}(h) = o_{\ava}^1(h) - o_{\ava}^2(h)$. 
\end{conjecture}

In the case of three-action games discrepancies do emerge; see Figure~\ref{fig:3-action_antagonistic_zero-sum}. These discrepancies appear bounded by planes. 

\begin{figure}[h!] 
    \centering
    \includegraphics[width=0.5\linewidth]{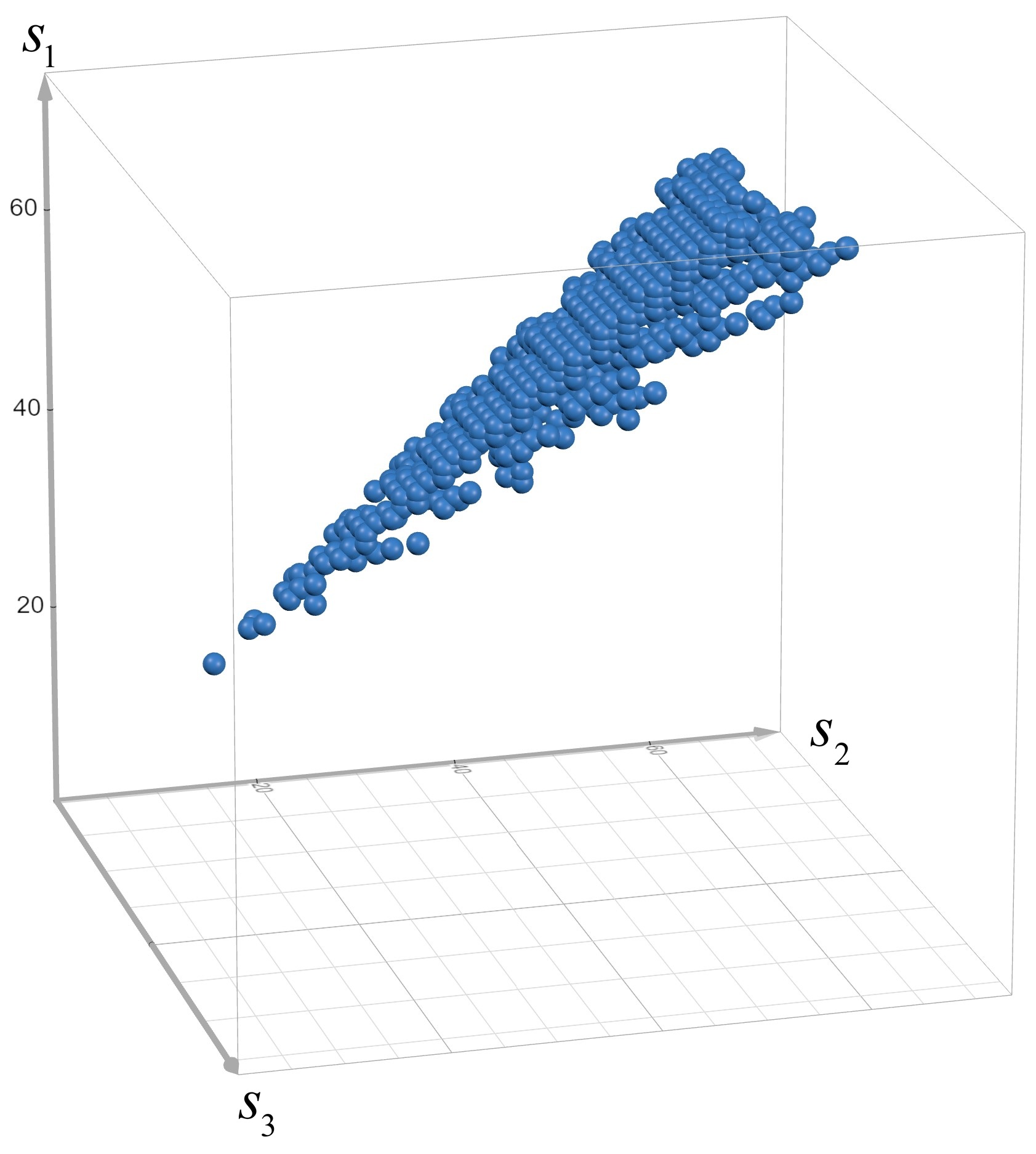}\caption{Points $(s_3, s_2, s_1)$ where there is a discrepancy between $\AvA$ and zero-sum games for $S = \{s_3, s_2, s_1\}$, with $s_3 < s_2 < s_1$. An online rotatable version can be found \href{https://www.desmos.com/3d/xbazasvgxs}{here}.}
    \label{fig:3-action_antagonistic_zero-sum}
\end{figure}
\begin{conjecture}Consider $S = \{s_3, s_2, s_1\}$ where $s_3 < s_2 < s_1$. Suppose that $s_2 \le 2s_3$ or $s_1 \ge s_2 + s_3$. Then there is no discrepancy between zero-sum scoring and antagonistic self-interest play.
\end{conjecture}

%We study these conjectures in a future paper \cite{Tanmay}.

\section{Some more open problems}\label{sec:future}
We have tested 200 random subtraction sets for each size $3\le |S|\le 10$, with $\max S \le 25$. For $h\le 300$, we have detected no positive outcome discrepancy, for either player, going from the setting FvF to AvA. This is in support of Conjecture~\ref{con:main}.

\begin{conjecture}
    For any $S$ and any two tie-breaking conventions, the discrepancies are bounded.
\end{conjecture}

\begin{problem}
 How large a discrepancy $\FvF\rightarrow \AvA$ (in absolute value) is possible as a function of $\max S$? 
\end{problem}

It is known that zero-sum cumulative subtraction games are eventually periodic, with a period of $2\max S$ (independent of the size of $S$). 
\begin{problem}
 Is it true that the discrepancies of self-interest subtraction are eventually periodic, with the period length $2\max S$? 
\end{problem}
\begin{problem}
Provide a bound on the pre-period length as a function of $\max S$.
\end{problem}

\begin{problem}
    What are the subtraction games with 0-discrepancy going from FvF to AvA? 
\end{problem}

Recall that by our notation, in the first case, the starting player deviates, while in the second case, the player who does not start deviates. So we have an asymmetric situation with respect to tie-breaking rules.
\begin{problem}
What are the subtraction games (and heap sizes) with positive discrepancy, going from FvF to AvF or FvA, respectively? 
\end{problem}
Recall what happened at heap size $61$ for the ruleset $S=\{4,5,9\}$ in FvF, a player who unilaterally deviates to antagonistic strictly gains. Experimental results point towards that this cannot happen if deviating towards AvA.
\begin{problem}
Is it true that no player can have a positive discrepancy by going from AvF or FvA to AvA?
\end{problem}
%We have looked into max/min (greedy/sacrifice) as alternative tie-breaking convetions for indifferent actions. But player utilities, for general subtraction sets, appear mixed up, with no clear preference order between pairs of tie-breaking rules, across all positions.
%\begin{problem}\label{prob:7}
    %Is there any pair of deterministic tie-breaking rules that bound all other ones, such that PSPE utility is maximized for both players across all positions?
%\end{problem}

%%%%%%%%%
\section{Discussion and related work}\label{sec:related}

 In classical combinatorial games, such as the normal-play convention, any indifference between options implies indifference for both players; perfect information induces unique optimal outcomes. By contrast, our self-interest model can produce local indifferences with discrepancies in the result for the other player. And thus, the choice of tie-breaking mechanism becomes essential in any pure subgame-perfect equilibrium (PSPE) computation. Our interest in staying close to combinatorial game theory lies in studying the deterministic backwards-induction algorithms induced by specific choices of deterministic tie-breaking mechanisms. 

Our approach has been to assign each player a deterministic tie-breaking identity, \emph{friendly} or \emph{antagonistic}, applied uniformly throughout the game tree. One may likewise impose a deterministic max/min (greedy/sacrifice) choice. For example, if $S=\{3,5\}$, then a greedy deterministic tie-breaking rule would be to always choose $5$ in case of indifference. This would allow backward induction to proceed unambiguously. Such a tie-breaking rule may seem simpler; however, the question of whether there is indifference between the options still requires a complete backward induction computation. The choice between friendly versus antagonistic (or greedy vs. sacrificial) tie-breaking depends on the entire continuation and hence requires full recursive evaluation. In this sense, either tie-breaking rule is both a behavioral commitment and a computational burden. In line with the tradition of combinatorial game theory, we assume perfect players capable of such reasoning. In a purely recreational setting, any self-interest subtraction game should, of course, be played without any such commitment.

%Either of these designs contrast with many settings in economics, AI, and multi-agent systems, where tie-breaking rules are typically exogenous, randomized, or fixed to promote fairness or efficiency. 
This perspective contrasts with many settings in economics, AI, and multi-agent systems, where tie-breaking rules are imposed exogenously, sometimes randomized to ensure fairness, or fixed for predictability or efficiency.

For instance, in a standard auction or matching model, if two agents submit identical bids or preferences, the model typically resolves the tie by randomization or by a fixed priority order specified in advance; the agents themselves do not control this rule, and it is not typically intended to affect the strategic structure beyond selecting one outcome among otherwise equivalent ones. Similarly, in many planning or learning algorithms, when multiple actions have equal value estimates, a fixed lexicographic rule is used to select one action, ensuring the algorithm is well-defined, rather than modelling any strategic preference of the agent. A related illustration arises in school choice: when applicants have the same score for a seat, a deterministic priority list systematically favors the same students, whereas random tie-breaking treats equally qualified applicants symmetrically.

In our setting, the tie-breaking rule is part of the player’s identity; it is not merely a local resolver of ambiguity, but prompts questions about comparative efficiency in a global sense, as different deterministic conventions can be compared across positions and instances. A natural question is whether there exists a pair of deterministic tie-breaking rules that jointly outperform any other pair for both players. We have tested both max/min and friendly/antagonistic, and we have not found any such pair. 

Across several domains, tie-breaking is often treated as a technicality, yet small local choices can propagate to global effects. In each of the following examples, a locally innocuous rule for resolving indifference alters the structure or outcomes of the model as a whole. In voting, different schemes (e.g., parallel-universes tiebreaking, PUT) alter strategies and collective outcomes~\cite{freeman2015}. In fair division, when multiple near-fair allocations exist, deterministic selection among them can shift welfare and perceived fairness~\cite{procaccia2017}. In tournaments, cycles can make winner selection underdetermined, and deterministic refinements such as minimal stable sets select different winner sets on different instances, so the choice of convention is not neutral~\cite{brandt2008}.  That said, the impact of tie-breaking is model-dependent: in well-behaved private-value auctions, equilibria can be invariant to the rule~\cite{jackson2005}. Overall, our results fit a broader pattern in which deterministic conventions for resolving indifference have system-level consequences, while also highlighting boundaries where such rules have no effect.

Our self-interest subtraction games exhibit a structural phenomenon. Deterministic tie-breaking is necessary to compute PSPE (well-posed recursion), but it is not always neutral. Empirically (see Appendix), we find subtraction sets for which some heaps make the supposedly indifferent player strictly better or worse under one tie-breaking rule than under another; no single rule dominates the others. As with elections and cascades, repeated local indifference can induce global divergences whenever $|S|>2$.

 In computational social choice, a voting rule must specify what happens when two (or more) candidates are exactly tied. Deterministic conventions such as lexicographic priority versus random-candidate tie-breaking are often treated as ``mere details''. However, recent work shows that the choice of tie-breaker has global consequences: it can change equilibria, manipulation incentives, and welfare. In particular, \cite{BaileyTovey2024TieBreakingManipulability} analyses standard rules and proves that neither deterministic convention dominates across instances: each can make elections more or less manipulable and can change price-of-anarchy bounds. Thus, global outcome can be sensitive to how we resolve local indifferences.

 In sequential social-learning models (wherein agents act one-by-one, each observing predecessors’ actions and a private signal), an \emph{informational cascade} arises when agents rationally ignore their signals and follow the crowd; learning may then fail even as more signals arrive \cite{BikhchandaniHirshleiferWelch1992}. A knife-edge occurs when public information and a private signal exactly balance, and the agent is indifferent; \cite{KoesslerZiegelmeyer2000TieBreakingCascades} shows that the existence and properties of cascades can hinge on this local tie-breaking convention. Relaxing the usual modelling convention that an indifferent agent selects a fixed action can admit equilibria without informational cascades. Again, a local choice can have global repercussions, namely, whether society asymptotically learns the true state. For example, in a sequential learning model, a cascade can arise at knife-edge histories because an indifferent agent is forced to select a predetermined action; allowing either choice at indifference can prevent the cascade and admit equilibria in which learning continues. 

Let us mention an area where the choice of a tie-breaking convention has surprising theoretical consequences. 
In Richman bidding combinatorial games players bid to gain the right to move (in an underlying combinatorial game), paying their bid to the opponent, and ties are resolved via a \emph{marker} that grants one player priority in case of equal bids. In the classical continuum setting, the theory is not sensitive to any particular tie-breaking convention~\cite{lazarus1996richman}, since equal bids occur with probability zero. In novel discrete settings, however, the marker becomes a strategic asset, influencing equilibrium behavior; see the survey~\cite{kant2025s}. Moreover, tie-breaking is central in \cite{kant2024, kant2025}, where games generalize alternating normal play (last-move-wins): players may include the marker in their bid to counter an impending zugzwang. Switching to random or one-sided tie-breaking would lead to a different (and probably less rich) theory. For example, even a modest change of ``not being allowed to include the marker in the bid'', induces mix strategies in equilibrium, which in itself appears to remove the standard CGT toolbox. 

Let us mention a related work on ``extensive-form perfect-information games'', where concept such as {\em credible threat} and {\em credible promise} are close to our antagonistic and firendly settings, although in our study we have not (yet) emphazised any ``threat'' component. 
%\paragraph{Extensive-form perfect-information games:}
Tran{\ae}s~\cite{tranaes} studies tie-breaking under indifference as a way to influence the opponent’s choices in finite extensive-form games of perfect information. A \emph{credible threat} selects, among tied continuations, a move that leaves the mover’s payoff unchanged while reducing the opponent’s continuation payoff, and thereby inducing earlier choices favorable to the mover. Symmetrically, a \emph{credible promise} selects, under indifference, a move that increases the opponent’s continuation payoff, again steering earlier choices. These conventions are incorporated into the equilibrium concept of \emph{path-perfect} subgame-perfect equilibrium: an equilibrium for which, along the path and in every subgame, there is no scope for indifference-based threats or promises to influence earlier choices profitably. Tran{\ae}s proves that every finite extensive-form game of perfect information admits a path-perfect pure subgame-perfect equilibrium.

\section*{Acknowledgements}
We would like to thank Eastside Preparatory School in Kirkland, WA, and their Independent Curriculum/Senior Thesis program for supporting this research and related work on subtraction games. Section~\ref{sec:related} has been composed by using also ChatGPT; we reach out to researchers in related areas to encourage overarching tie-breaking studies in future projects.

\bibliographystyle{abbrvnat}   % or abbrvnat, unsrtnat, etc.
\bibliography{ref_selfint}            % refs.bib

\begin{thebibliography}{17}
\providecommand{\natexlab}[1]{#1}
\providecommand{\url}[1]{\texttt{#1}}
\expandafter\ifx\csname urlstyle\endcsname\relax
  \providecommand{\doi}[1]{doi: #1}\else
  \providecommand{\doi}{doi: \begingroup \urlstyle{rm}\Url}\fi

\bibitem[Bailey and Tovey(2024)]{BaileyTovey2024TieBreakingManipulability}
J.~P. Bailey and C.~A. Tovey.
\newblock The impact of tie-breaking on the manipulability of elections.
\newblock In \emph{Proceedings of the 23rd International Conference on Autonomous Agents and Multiagent Systems (AAMAS 2024)}, pages 105--113. IFAAMAS, 2024.

\bibitem[Berlekamp et~al.(2004)Berlekamp, Conway, and Guy]{BCG1982}
E.~R. Berlekamp, J.~H. Conway, and R.~K. Guy.
\newblock \emph{Winning Ways for Your Mathematical Plays}.
\newblock A K Peters/CRC Press, 2004.

\bibitem[Bikhchandani et~al.(1992)Bikhchandani, Hirshleifer, and Welch]{BikhchandaniHirshleiferWelch1992}
S.~Bikhchandani, D.~Hirshleifer, and I.~Welch.
\newblock A theory of fads, fashion, custom, and cultural change as informational cascades.
\newblock \emph{Journal of Political Economy}, 100\penalty0 (5):\penalty0 992--1026, 1992.

\bibitem[Brandt(2011)]{brandt2008}
F.~Brandt.
\newblock Minimal stable sets in tournaments.
\newblock \emph{Journal of Economic Theory}, 146\penalty0 (4):\penalty0 1481--1499, 2011.

\bibitem[Cohensius et~al.(2019)Cohensius, Larsson, Meir, and Wahlstedt]{CLMW2019}
G.~Cohensius, U.~Larsson, R.~Meir, and D.~Wahlstedt.
\newblock Cumulative subtraction games.
\newblock \emph{Electronic Journal of Combinatorics}, 26\penalty0 (4):\penalty0 P4.52, 2019.
\newblock \doi{10.37236/7904}.

\bibitem[Freeman et~al.(2015)Freeman, Brill, and Conitzer]{freeman2015}
R.~Freeman, M.~Brill, and V.~Conitzer.
\newblock General tiebreaking schemes for computational social choice.
\newblock In \emph{Proceedings of the 14th International Conference on Autonomous Agents and Multiagent Systems (AAMAS 2015)}, pages 1401--1409, 2015.

\bibitem[Jackson and Swinkels(2005)]{jackson2005}
M.~O. Jackson and J.~M. Swinkels.
\newblock Existence of equilibrium in single and double private value auctions.
\newblock \emph{Econometrica}, 73\penalty0 (1):\penalty0 93--139, 2005.

\bibitem[Kant and Larsson(2025)]{kant2025s}
P.~Kant and U.~Larsson.
\newblock Survey on {R}ichman bidding combinatorial games.
\newblock In U.~Larsson, editor, \emph{Games of No Chance 6}, volume~71 of \emph{MSRI/SLMath Publications}. Cambridge University Press, 2025.

\bibitem[Kant et~al.(2024)Kant, Larsson, Rai, and Upasany]{kant2024}
P.~Kant, U.~Larsson, R.~K. Rai, and A.~V. Upasany.
\newblock Bidding combinatorial games.
\newblock \emph{Electronic Journal of Combinatorics}, 31\penalty0 (1):\penalty0 P1.51, 2024.

\bibitem[Kant et~al.(2025)Kant, Larsson, Rai, and Upasany]{kant2025}
P.~Kant, U.~Larsson, R.~K. Rai, and A.~V. Upasany.
\newblock Constructive comparison in bidding combinatorial games.
\newblock \emph{International Journal of Game Theory}, 54\penalty0 (1):\penalty0 1--28, 2025.

\bibitem[Koessler and Ziegelmeyer(2000)]{KoesslerZiegelmeyer2000TieBreakingCascades}
F.~Koessler and A.~Ziegelmeyer.
\newblock Tie-breaking rules and informational cascades: A note.
\newblock Technical Report Working Paper 2000-09, BETA, Universit{\'e} Louis Pasteur (Strasbourg I), 2000.

\bibitem[Kulkarni and Larsson()]{Tanmay}
T.~Kulkarni and U.~Larsson.
\newblock Discrepancies in various 2-player self-interest cumulative subtraction games.
\newblock In preparation. Online figures: \url{https://www.desmos.com/calculator/ynbxlvaatq}, \url{https://www.desmos.com/3d/xbazasvgxs}, \url{https://www.desmos.com/3d/pqirkkmvjq}.

\bibitem[Kurokawa et~al.(2017)Kurokawa, Procaccia, and Wang]{procaccia2017}
D.~Kurokawa, A.~D. Procaccia, and J.~Wang.
\newblock Fair enough: Guaranteeing approximate maximin shares.
\newblock \emph{Journal of the ACM}, 64\penalty0 (5):\penalty0 1--27, 2017.

\bibitem[Larsson and Saha(2025)]{LS2025}
U.~Larsson and I.~Saha.
\newblock A brief conversation about subtraction games.
\newblock In U.~Larsson, editor, \emph{Games of No Chance 6}, volume~71 of \emph{MSRI/SLMath Publications}. Cambridge University Press, 2025.

\bibitem[Larsson et~al.(2020)Larsson, Meir, and Zick]{LMZ}
U.~Larsson, R.~Meir, and Y.~Zick.
\newblock Cumulative games: Who is the current player?, 2020.
\newblock URL \url{https://arxiv.org/abs/2005.06326}.

\bibitem[Lazarus et~al.(1996)Lazarus, Loeb, Propp, and Ullman]{lazarus1996richman}
A.~J. Lazarus, D.~E. Loeb, J.~G. Propp, and D.~Ullman.
\newblock Richman games.
\newblock In R.~J. Nowakowski, editor, \emph{Games of No Chance}, volume~29 of \emph{MSRI Publications}, pages 439--449. Cambridge University Press, 1996.

\bibitem[Tran{\ae}s(1998)]{tranaes}
T.~Tran{\ae}s.
\newblock Tie-breaking in games of perfect information.
\newblock \emph{Games and Economic Behavior}, 22:\penalty0 148--161, 1998.

\end{thebibliography}


\begin{thebibliography}{99}

\bibitem{aumann1959} 
R. J. Aumann. 
\newblock Acceptable points in general cooperative $n$-person games. 
\newblock In {\em Contributions to the Theory of Games IV}, pages 287--324, Princeton University Press, 1959.

\bibitem{blume1991} L. Blume, A. Brandenburger, and E. Dekel. Lexicographic probabilities and choice under uncertainty. \emph{Econometrica}, 59(1):61--79, 1991.

\bibitem{brandt2016}
F. Brandt, V. Conitzer, and U. Endriss.
\newblock Computational social choice.
\newblock In *Handbook of Computational Social Choice*, pages 3--38. Cambridge University Press, 2016.


\bibitem{BCG1982} 
E. R. Berlekamp, J. H. Conway and R. K. Guy. 
\newblock {\em Winning Ways for Your Mathematical Plays}. 
\newblock AK Peters/CRC Press, 2004.

\bibitem{CLMW2019} 
G. Cohensius, U. Larsson, R. Meir, and D. Wahlstedt. 
\newblock Cumulative subtraction games. 
\newblock {\em Electron. J. Combin.}, 26, (2019).

\bibitem{devanur2007} N. R. Devanur, C. H. Papadimitriou, A. Saberi, and V. Vazirani. Market equilibrium via a primal–dual algorithm for a convex program. \emph{Journal of the ACM}, 55(5):1–18, 2008.

\bibitem{fudenberg} 
D. Fudenberg and J. Tirole. 
\newblock {\em Game Theory}. 
\newblock MIT Press, 1991.

%\bibitem{hansen2004} E. A. Hansen, D. S. Bernstein, and S. Zilberstein. Dynamic programming for partially observable stochastic games. \emph{Proceedings of the AAAI Conference on Artificial Intelligence}, 2004.

\bibitem{hansen2004}
E.~A.~Hansen, D.~S.~Bernstein, and S.~Zilberstein.
\newblock Dynamic programming for partially observable stochastic games.
\newblock In \emph{Proceedings of the 19th AAAI Conference on Artificial Intelligence (AAAI)}, pages 709--715, 2004.

\bibitem{hansen2004}
E.~A.~Hansen, D.~S.~Bernstein, and S.~Zilberstein.
\newblock Dynamic programming for partially observable stochastic games.
\newblock In \emph{Proceedings of the 19th AAAI Conference on Artificial Intelligence (AAAI)}, pages 709--715, 2004.
\newblock \url{https://cdn.aaai.org/AAAI/2004/AAAI04-112.pdf}

\bibitem{hansen2004}
E.~A.~Hansen, D.~S.~Bernstein, and S.~Zilberstein.
\newblock Dynamic programming for partially observable stochastic games.
\newblock In \emph{Proceedings of the 19th AAAI Conference on Artificial Intelligence (AAAI)}, pages 709--715, 2004.
\newblock \href{https://cdn.aaai.org/AAAI/2004/AAAI04-112.pdf}{\nolinkurl{aaai.org/AAAI04-112.pdf}}


\bibitem{LMZ} 
U. Larsson, R. Meir, Y. Zick. 
\newblock Cumulative games: Who is the current player? 
\newblock Preprint arXiv:2005.06326.

\bibitem{LS2025} 
U. Larsson, I. Saha. 
\newblock A brief conversation about subtraction games. 
\newblock In (Ed. U. Larsson) {\em Games of No Chance 6}, MSRI/SLMath Publ. 71, Cambridge University Press, 2025.

\bibitem{moulin1986} 
H. Moulin. 
\newblock Choosing from a tournament. 
\newblock {\em Social Choice and Welfare}, 3(4):271--291, 1986.

\bibitem{osborne} 
M. J. Osborne and A. Rubinstein. 
\newblock {\em A Course in Game Theory}. 
\newblock MIT Press, 1994.

\bibitem{procaccia2017} A. D. Procaccia and J. Wang. Fair enough: Guaranteeing approximate maximin shares. \emph{Journal of the ACM}, 64(5):1–27, 2017.

\bibitem{Tanmay} 
T. Kulkarni and U. Larsson. 
\newblock Discrepancies in various 2-player self-interest cumulative subtraction games. 
\newblock In preparation. 
\newblock Online figures: \url{https://www.desmos.com/calculator/erdndez4ey}, \url{https://www.desmos.com/3d/xbazasvgxs}, and \url{https://www.desmos.com/3d/38511dd74c}.

\end{thebibliography}

\begin{thebibliography}{99}

% local vs global
\bibitem{BaileyTovey2024TieBreakingManipulability}
Bailey, J. P., \& Tovey, C. A. (2024).
\textit{The Impact of Tie-Breaking on the Manipulability of Elections}.
In \textit{Proceedings of the 23rd International Conference on Autonomous Agents and Multiagent Systems (AAMAS 2024)}.
Richland, SC: IFAAMAS.

\bibitem{KoesslerZiegelmeyer2000TieBreakingCascades}
Koessler, F., \& Ziegelmeyer, A. (2000).
\textit{Tie-Breaking Rules and Informational Cascades: A Note}.
BETA Working Paper No.\ 2000-09, Université Louis Pasteur (Strasbourg I).

\bibitem{BikhchandaniHirshleiferWelch1992}
Bikhchandani, S., Hirshleifer, D., \& Welch, I. (1992).
A theory of fads, fashion, custom, and cultural change as informational cascades.
\textit{Journal of Political Economy}, \textit{100}(5), 992–1026.


% --- Voting ---

\bibitem{freeman2015}
R.~Freeman, M.~Brill, and V.~Conitzer.
\newblock General tiebreaking schemes for computational social choice.
\newblock In \emph{Proceedings of the 14th International Conference on Autonomous Agents and Multiagent Systems (AAMAS)}, pages 1401--1409, 2015.

\bibitem{obraztsova2011}
Obraztsova, S., Elkind, E., \& Hazon, N. (2011).
Ties matter: Complexity of voting manipulation revisited.
In \textit{Proceedings of the 10th International Conference on Autonomous Agents and Multiagent Systems (AAMAS 2011)} (pp.~71--78).
Taipei, Taiwan: IFAAMAS.

\bibitem{elkind2011}
E.~Elkind and S.~Obraztsova.
\newblock On the complexity of voting manipulation under randomized tie-breaking.
\newblock In \emph{Proceedings of the 22nd International Joint Conference on Artificial Intelligence (IJCAI)}, pages 319--324, 2011.

\bibitem{mattei2013}
N.~Mattei, N.~Narodytska, and T.~Walsh.
\newblock How hard is it to control an election by breaking ties?
\newblock In \emph{Proceedings of the 23rd International Joint Conference on Artificial Intelligence (IJCAI)}, pages 258--264, 2013.

\bibitem{bailey2024}
J.~Bailey and C.~Tovey.
\newblock Impact of tie-breaking on the manipulability of elections.
\newblock In \emph{Proceedings of the 23rd International Conference on Autonomous Agents and Multiagent Systems (AAMAS)}, pages 105--113, 2024.

% --- Fair division ---

\bibitem{procaccia2017}
A.~D. Procaccia and J.~Wang.
\newblock Fair enough: Guaranteeing approximate maximin shares.
\newblock In \emph{Proceedings of the 15th ACM Conference on Economics and Computation (EC)}, pages 675--692, 2014. % corrected venue/year

\bibitem{aziz2017}
H.~Aziz, I.~Caragiannis, A.~Igarashi, and T.~Walsh.
\newblock Fair allocation of indivisible goods and chores.
\newblock In \emph{Proceedings of the 28th International Joint Conference on Artificial Intelligence (IJCAI)}, 2019. % corrected authors/year/venue

\bibitem{narosky2022}
A.~Narosky, J.~Spiritus, and A.~Szymczak.
\newblock Tie-breaking mechanisms in assignment problems: Randomness, strategyproofness, and fairness.
\newblock \emph{European Economic Review}, 147:104156, 2022.


% --- Tournaments ---

\bibitem{brandt2008}
F.~Brandt.
\newblock Minimal stable sets in tournaments.
\newblock \emph{Journal of Economic Theory}, 146(4):1481--1499, 2011. % use journal version; keeps key

\bibitem{brandt2016tournament}
F.~Brandt, M.~Brill, and P.~Harrenstein.
\newblock Tournament solutions.
\newblock In F.~Brandt, V.~Conitzer, U.~Endriss, J.~Lang, and A.~D. Procaccia (Eds.), \emph{Handbook of Computational Social Choice}, chapter 3, pages 57--84. Cambridge University Press, 2016.

% --- Lexicographic / games ---

\bibitem{chatterjee2023}
K.~Chatterjee, J.-P. Katoen, M.~Weininger, and T.~Winkler.
\newblock Stochastic games with lexicographic objectives.
\newblock \emph{Formal Methods in System Design}, 59(1):6--47, 2023.

\bibitem{aghajohari2021}
M.~Aghajohari, G.~Avni, and T.~A. Henzinger.
\newblock Determinacy in discrete-bidding infinite-duration games.
\newblock \emph{Logical Methods in Computer Science}, 17(1):1--23, 2021.

\bibitem{kolpin1993core}
V.~Kolpin.
\newblock The core of discontinuous games.
\newblock \emph{International Journal of Game Theory}, 22:43--50, 1993.

% --- Auctions / contests ---

\bibitem{jackson2005}
M.~O. Jackson and J.~M. Swinkels.
\newblock Existence of equilibrium in single and double private value auctions.
\newblock \emph{Econometrica}, 73(1):93--139, 2005.

\bibitem{chen2023}
X.~Chen and B.~Peng.
\newblock Complexity of equilibria in first-price auctions under general tie-breaking rules.
\newblock In \emph{Proceedings of the 55th Annual ACM Symposium on Theory of Computing (STOC)}, pages 1645--1658, 2023.

\bibitem{llorente2023}
A.~Llorente-Saguer, R.~M. Sheremeta, and N.~Szech.
\newblock Designing contests between heterogeneous contestants: An experimental study of tie-breaks and bid-caps in all-pay auctions.
\newblock \emph{European Economic Review}, 154:104327, 2023.

% --- Mechanism design / matching ---

\bibitem{ashlagi2020}
I.~Ashlagi and A.~Nikzad.
\newblock What matters in school choice tie\mbox{-}breaking? How competition guides design.
\newblock \emph{Journal of Economic Theory}, 190:105120, 2020.

\bibitem{ehlers2018}
L.~Ehlers.
\newblock Strategy-proof tie-breaking in matching with priorities.
\newblock \emph{Theoretical Economics}, 13(1):173--197, 2018.

\bibitem{imamura2023}
K.~Imamura and K.~Tomoeda.
\newblock Tie-breaking or not: A choice function approach.
\newblock Working paper, 2023.

% --- Combinatorial auctions / exchanges ---

\bibitem{conitzer2002}
V.~Conitzer and T.~Sandholm.
\newblock Complexity of mechanism design.
\newblock In \emph{Proceedings of the 18th Conference on Uncertainty in Artificial Intelligence (UAI)}, pages 103--110, 2002. % corrected to verified paper

\bibitem{othman2011}
R.~W. Day and P.~Cramton.
\newblock Quadratic core-selecting payment rules for combinatorial auctions.
\newblock \emph{Operations Research}, 60(3):588--603, 2012. % replace unverifiable Othman item with canonical core-selecting reference; keeps key


% --- Routing / search ---

\bibitem{asai2017}
M.~Asai.
\newblock Tie-breaking strategies for cost-optimal best-first search.
\newblock \emph{Journal of Artificial Intelligence Research}, 58:67--121, 2017. % corrected volume/pages

\bibitem{correa2018}
D.~Corr\^ea, A.~Hern\'andez, V.~Costa, and M.~Veloso.
\newblock Analyzing tie-breaking strategies for the A$^*$ algorithm.
\newblock In \emph{Proceedings of the 27th International Joint Conference on Artificial Intelligence (IJCAI)}, pages 4723--4729, 2018.

\bibitem{jacobs2012}
T.~Jacobs.
\newblock Analytical aspects of tie breaking.
\newblock \emph{Theoretical Computer Science}, 459:74--91, 2012.

\bibitem{bodwin2021}
G.~Bodwin and M.~Parter.
\newblock Restorable shortest path tiebreaking for edge-faulty graphs.
\newblock \emph{Journal of the ACM}, 70(5):28:1--28:32, 2023. % update to journal version; keeps key

% --- ML / RL text ---

\bibitem{sutton2018}
R.~S.~Sutton and A.~G.~Barto.
\newblock \emph{Reinforcement Learning: An Introduction}.
\newblock MIT Press, second edition, 2018.

%%%General Game Theory references
\bibitem{aumann1959} 
R.~J.~Aumann.
\newblock Acceptable points in general cooperative $n$-person games.
\newblock In \emph{Contributions to the Theory of Games IV}, pages 287--324. Princeton University Press, 1959.

\bibitem{fudenberg} 
D.~Fudenberg and J.~Tirole.
\newblock \emph{Game Theory}.
\newblock MIT Press, 1991.

\bibitem{osborne} 
M.~J.~Osborne and A.~Rubinstein.
\newblock \emph{A Course in Game Theory}.
\newblock MIT Press, 1994.

% Extensive form games

\bibitem{tranaes}
Torben Tranæs.
\newblock Tie-breaking in games of perfect information.
\newblock \emph{Games and Economic Behavior},
22:148-161, 1998.

% Stochastic games

%\bibitem{hansen2004}
%E.~A.~Hansen, D.~S.~Bernstein, and S.~Zilberstein.
%\newblock Dynamic programming for partially observable stochastic games.
%\newblock In \emph{Proceedings of the 19th AAAI Conference on Artificial Intelligence (AAAI)}, pages 709--715, 2004.


%Richman
\bibitem{kant2025s}
P. Kant, U. Larsson,
\newblock Survey on bidding combinatorial games.
\newblock in {\em Games of No Chance 6}, 2025
\newblock MSRI/SLMath Volume 71, Cambridge University Press.

\bibitem{kant2024}
P. Kant, U. Larsson, R. K. Rai, and A. V. Upasany, 
\newblock Bidding combinatorial games, 
\newblock {\em Electron. J. Combin.}  31 (2024) P1.51.

\bibitem{kant2025}
P. Kant, U. Larsson, R. K. Rai and A. V. Upasany, 
\newblock Constructive comparison in bidding combinatorial games, 
\newblock {\em Int. J. Game Theory} (2025).

\bibitem{lazarus1996richman}
A.~J. Lazarus, D.~E. Loeb, J.~G. Propp, and D.~Ullman.
\newblock Richman games.
\newblock In R.~J. Nowakowski, editor, \emph{Games of No Chance}, MSRI
  Publications, volume 29, pages 439--449. Cambridge University Press, 1996.

\bibitem{lazarus1999auction}
A.~J. Lazarus, D.~E. Loeb, J.~G. Propp, W.~R. Stromquist, and D.~H. Ullman.
\newblock Combinatorial games under auction play.
\newblock \emph{Games and Economic Behavior}, 27(2):229--264, 1999.

\bibitem{payne2009hex}
S.~Payne and E.~Robeva.
\newblock Artificial intelligence for Bidding Hex.
\newblock In R.~J. Nowakowski, editor, \emph{Games of No Chance 4}, MSRI
  Publications, volume 63, pages 207--214. Cambridge University Press, 2009.

\bibitem{larsson2019bchess}
U.~Larsson and J.~W{\"a}stlund.
\newblock Endgames in bidding chess.
\newblock In (Ed. U.~Larsson) \emph{Games of No Chance 5}, MSRI
  Publications, volume 70, pages 421--438. Cambridge University Press, 2019.

\bibitem{develin2010}
M.~Develin and S.~Payne.
\newblock Discrete bidding games.
\newblock \emph{Electronic Journal of Combinatorics}, 17(1):R85, 2010.


\bibitem{BCG1982} E. R. Berlekamp, J. H. Conway and R. K. Guy. {\em Winning Ways for Your Mathematical Plays}. AK Peters/CRC Press, 2004.
 %E. Berlekamp, J. H. Conway, R. Guy, {\em Winning Ways, for Your Mathematical Plays}, 1982
\bibitem{CLMW2019} G. Cohensius, U. Larsson, R. Meir, and D. Wahlstedt, Cumulative subtraction games, {\em Electron. J. Combin.} 26, (2019).  
\bibitem{LMZ} U. Larsson, R. Meir, Y. Zick, Cumulative games: Who is the current player?, preprint arXiv:2005.06326.
\bibitem{LS2025} U. Larsson, I. Saha, A brief conversation about subtraction games; in (Ed. U. Larsson) {\em Games of No Chance 6}, MSRI/SLMath Publ. 71, Cambridge University Press, 2025.
\bibitem{Tanmay} T. Kulkarni and U. Larsson, Discrepancies in various 2-player self-interest cumulative subtraction games. In preparation. Online figures: https://www.desmos.com/calculator/ynbxlvaatq, https://www.desmos.com/3d/xbazasvgxs and https://www.desmos.com/3d/pqirkkmvjq
\end{thebibliography}

\clearpage
\section*{Appendix: Unilateral deviation for larger subtraction sets}\label{sec:unidev}
% Let us study some tables with the initial outcomes for all tie-breaking combinations of the rulesets 
% \begin{itemize}
%     \item $S_1=\{3,8,11,13\}$ (Table~\ref{tab:3_8_11_13}) and
%     \item  $S_2=\{4,5,9\}$ (Table~\ref{tab:deviation459}). 
% \end{itemize}
% %The outcome and discrepancy tables are provided in Tables~\ref{tab:3_8_11_13} and \ref{tab:deviation459}. 
% In particular, we focus on discrepancies that do not appear in the two-action case, and explain how they arise through backward induction. Aiding the reader we include the PSPE-moves in the same columns as the outcomes.'
%\texttt{
Our fish-pile-wise Penguins have learned that acting friendly in cases of indifference, through a sacrifice, can benefit the other player. This was evident in Table~\ref{tab:3_5}, where the available slack was used precisely for this purpose. Due to the recursive definition of outcomes, other curious effects await our friends, the Penguins. \vspace{2 mm}

\begingroup
\sffamily
\begin{quote} %

Suppose Alice and Bob are faced with a heap of size 61, from which they may grab 4, 5, or 9 fishes. Bob has promised his Penguin clan to return with a total of 28 fishes. Alice, as the starting player, has pledged to remain friendly throughout. The question arises: should Bob maintain his friendly disposition, or should he deviate and act antagonistically? 
In this particular instance, it turns out that deviating benefits him. By becoming antagonistic, Bob secures one fish for immediate consumption, warming his beak, while still delivering the promised 28 fishes to his clan. If he remains friendly, however, he receives nothing for himself. Bob is puzzled, for his computational skills do not stretch beyond one heap of fish at a time: his tie-breaking rule is designed to leave him indifferent in such cases, yet here he is clearly rewarded for turning against Alice, who correspondingly loses a fish. %Bob is puzzled, while his computational power does not stretch beyond one heap of fish at the time. His tie-breaking rule is designed to leave him indifferent in such cases, yet here he is clearly rewarded for turning against Alice, who correspondingly loses a fish.
The next day, they return to the fish pile, but the rules have changed. On their turns, they must now take 3, 8, 11, or 13 fishes, starting from a heap of size 36.  Remembering the previous day’s outcome, and despite being fish-pile-wise, Bob is tempted by hunger. He once again turns antagonistic. 
However, this time the result is unfavourable. Bob’s clan only requires 14 fishes, knowing that Alice will start, but his tie-breaking choice leads to a total of only 13 fishes, which is one short of the requirement. His reputation within the clan is tarnished, and his beak remains unsatisfied.
%}
\end{quote}
\par
\endgroup
\vspace{2 mm}

%Our all-knowing Penguins have learned that acting friendly in case of indifference, via a sacrifice, can benefit the other player; we saw this in Table~\ref{tab:3_5}, where the slack has been used to this purpose. Because of the recursive definition of outcomes, other curious effects await our friends, the Penguins. Suppose Alice and Bob faces a heap of size 61, where they may grab, 4,5 or 9 fishes. Bob has promised his Penguin clan to return with 28 fishes. Alice, as a starting player, has promised to remain friendly. Now, should he deviate and become an antagonistic Penguin? It turns out that, in this setting he will benefit a fish to consume immediately and warm hos beak, while the clan will receive their fair share. If he remeins friendly then he will get nothing for himself. Bob is puzzled though, becuase while his tie-breaking behavior is defined to leave him indifferent, in this particular case he was awarded a fish, by turning against Alice (who does lose a fish in this case). Next they, they return to the fish pile, but rules have changed, at their turn, they must take 3,8,11 or 13 fish, starting from a heap of size 36. Bob's clan only require 14 fish (knowing that Alice starts). Bob recalls his yesterday experiences, and allthough being all-knowing makes a misjudgment, his hungry beak indulges him in once again turning antagonistic. But behold this leads to only 13 fish, and his reputation with the clan is tarnished, while his beak remains unsatisfied. 

In order to understand what happened, let us study the initial outcomes for all tie-breaking combinations of the following rulesets:
\begin{itemize}

\item $S_1=\{3,8,11,13\}$ (see Table~\ref{tab:3_8_11_13}), and
\item $S_2=\{4,5,9\}$ (see Table~\ref{tab:deviation459}).
\end{itemize}
We will reveal discrepancies that do not occur in the two-action case and explain how they arise through backward induction. For the reader’s convenience, the tables list the PSPE moves in the same columns as the outcomes.

%In particular, we will focus on describing some type of discrepancies that do not appear in the two-action case, and how they arise from the backward induction computation. 

\subsection*{The ruleset $S_1=\{3,8,11,13\}$}
Recall that in the two-action setting in FvF, a player who deviates unilaterally cannot harm their own utility (unless the opponent also deviates and becomes antagonistic). Consider the ruleset $S_1$, with outcome discrepancies and optimal moves for all heap sizes smaller than $50$ listed in Table~\ref{tab:3_8_11_13}. For all heap sizes below $28$, PSPE utilities are independent of tie-breaking conventions. At heap size $28$, however, if Alice deviates from FvF to AvF, then Bob loses one unit: in AvF, Alice chooses $3$ instead of her FvF choice $11$.

This leads to a curious effect at a heap size of $36$, namely that Bob loses a unit by becoming antagonistic, that is, by playing in the FvA setting instead of remaining friendly in FvF. Let us justify this phenomenon, since it shows how deterministic tie-breaking becomes much more intricate in multi-action games than in the simpler two-action case. In the FvF setting, with $h=36$, Alice plays $8$, producing the outcome $(22,14)$. But in the FvA setting, while remaining friendly, she does not choose $8$, since that yields her utility $8+13=21$, whereas playing $11$ yields $11+11=22$. Conversely, in the FvF setting, Alice does not choose $11$, which would lead to Bob's utility $13$ (the first player's PSPE utility at a heap of size $25$), because by instead reducing the heap by $8$, she ensures that Bob receives $14$ units.

\subsection*{The ruleset $S_2=\{4,5,9\}$}
The ruleset $S_2$ contains a diametral property of $S_1$, namely, here a player can strictly benefit by unilateral deviation to an antagonistic. As Lemma~\ref{lem:fvfdeviate} shows, this is impossible whenever $|S|=2$. For $S_2$, however, such a case occurs at heap size $h=61$. 
By Alice removing $4$ in the FvA setting, instead of $5$ or $9$ (the PSPE moves in the FvF setting), Bob both adds one unit to his own PSPE utility, and removes one unit from Alice's. 
%When Alice removes $4$ in the FvA setting, instead of $5$ or $9$ (the PSPE moves in FvF), then Bob gains one unit in his PSPE utility while that simultaneously reduces Alice’s utility by one. 
The PSPE utilities of all options at $h=61$ are: 
%Let us list the PSPE utilities of all options of $h=61$ in both settings: 
$$o_{\fvf}(h-4)=(29,28), o_{\fvf}(h-5)=(28,28), o_{\fvf}(h-9)=(28,24) $$ and $$o_{\avf}(h-4)=(29,28), o_{\avf}(h-5)=(28,27), o_{\avf}(h-9)=(28,23).$$  

In the FvA setting, Alice observes that all three moves give her the same utility of $32$. Acting friendly, she therefore chooses the option that maximizes Bob’s outcome, namely removing $4$, which gives him $29$ rather than $28$. By contrast, in FvF the situation is entirely different: here she can obtain $33$ by playing either $5$ or $9$, while playing $4$ yields only $32$. In FvF, there is no indifference that affects Bob's PSPE utility, and there is no reason for her to sacrifice by playing $4$. This explains the interesting discrepancies that arise in the FvA setting at $h=61$ (and many more at larger heaps). It is also interesting that whenever the deviating player gets a higher utility, the utility of the other player reduces by the same amount. This is unlike the case when the deviating player gets a lower utility, while the utility of the other player remains the same.
Purely self-interested players, as we define in this paper, aim solely to maximize their own utility, while disregarding that of their opponent. This example illustrates that a self-interested player may be strictly motivated to deviate and become antagonistic; this is a non-trivial observation, while the defined outcomes do not reveal any such incentive. With only two actions, no such intrinsic motivation exists, as established earlier in this paper. %: unilateral tie-breaking deviations from friendly to antagonistic are monotonely non-increasing the oth. 

In Tables~\ref{tab:3_8_11_13} and \ref{tab:deviation459} the $\delta$ symbols represent the deterministic tie-breaking deviations from FvF for the respective settings and players, where $\alpha$ is Alice and $\beta$ is Bob.

%In a pure self-interest convention, players are always maximizing their own utility, but disregarding their opponent's ditto. So, this example  illustrates that self-interested players can be strictly motivated to deviate. Whenever there are only two actions, there is no intrinsic motivation to deviate, which this paper justifies. 
%Let us explain the behavior of $S_2$ at the heap of size 61 by looking into the second page of the table in Appendix~2. Again, as in the case of $S_1$, it is the second player, Bob, whose discrepancy changes. 

%Consider first the heap of size 28. , which . 

%\clearpage
\begin{table}\caption{The initial outcomes, together with PSPE moves for the ruleset $S = \{3, 8, 11, 13\}$. }\label{tab:3_8_11_13}
\begin{center}
{\small  
\begin{tabular}{|r|l|l|r|r|l|r|r|l|r|r|}
\hline
Heap & FvF  & FvA  & $\delta_\alpha$ & $\delta_\beta$ & AvF & $\delta_\alpha$ & $\delta_\beta$ & AvA & $\delta_\alpha$ & $\delta_\beta$ \\
\hline
0 & (0,0) & (0,0) & 0 & 0 & (0,0) & 0 & 0 & (0,0) & 0 & 0 \\
1 & (0,0) & (0,0) & 0 & 0 & (0,0) & 0 & 0 & (0,0) & 0 & 0 \\
2 & (0,0) & (0,0) & 0 & 0 & (0,0) & 0 & 0 & (0,0) & 0 & 0 \\
3 & (3,0), 3 & (3,0), 3 & 0 & 0 & (3,0), 3 & 0 & 0 & (3,0), 3 & 0 & 0 \\
4 & (3,0), 3 & (3,0), 3 & 0 & 0 & (3,0), 3 & 0 & 0 & (3,0), 3 & 0 & 0 \\
5 & (3,0), 3 & (3,0), 3 & 0 & 0 & (3,0), 3 & 0 & 0 & (3,0), 3 & 0 & 0 \\
6 & (3,3), 3 & (3,3), 3 & 0 & 0 & (3,3), 3 & 0 & 0 & (3,3), 3 & 0 & 0 \\
7 & (3,3), 3 & (3,3), 3 & 0 & 0 & (3,3), 3 & 0 & 0 & (3,3), 3 & 0 & 0 \\
8 & (8,0), 8 & (8,0), 8 & 0 & 0 & (8,0), 8 & 0 & 0 & (8,0), 8 & 0 & 0 \\
9 & (8,0), 8 & (8,0), 8 & 0 & 0 & (8,0), 8 & 0 & 0 & (8,0), 8 & 0 & 0 \\
10 & (8,0), 8 & (8,0), 8 & 0 & 0 & (8,0), 8 & 0 & 0 & (8,0), 8 & 0 & 0 \\
11 & (11,0), 11 & (11,0), 11 & 0 & 0 & (11,0), 11 & 0 & 0 & (11,0), 11 & 0 & 0 \\
12 & (11,0), 11 & (11,0), 11 & 0 & 0 & (11,0), 11 & 0 & 0 & (11,0), 11 & 0 & 0 \\
13 & (13,0), 13 & (13,0), 13 & 0 & 0 & (13,0), 13 & 0 & 0 & (13,0), 13 & 0 & 0 \\
14 & (13,0), 13 & (13,0), 13 & 0 & 0 & (13,0), 13 & 0 & 0 & (13,0), 13 & 0 & 0 \\
15 & (13,0), 13 & (13,0), 13 & 0 & 0 & (13,0), 13 & 0 & 0 & (13,0), 13 & 0 & 0 \\
16 & (13,3), 13 & (13,3), 13 & 0 & 0 & (13,3), 13 & 0 & 0 & (13,3), 13 & 0 & 0 \\
17 & (14,3), 11 & (14,3), 11 & 0 & 0 & (14,3), 11 & 0 & 0 & (14,3), 11 & 0 & 0 \\
18 & (14,3), 11 & (14,3), 11 & 0 & 0 & (14,3), 11 & 0 & 0 & (14,3), 11 & 0 & 0 \\
19 & (16,3), 13 & (16,3), 13 & 0 & 0 & (16,3), 13 & 0 & 0 & (16,3), 13 & 0 & 0 \\
20 & (16,3), 13 & (16,3), 13 & 0 & 0 & (16,3), 13 & 0 & 0 & (16,3), 13 & 0 & 0 \\
21 & (13,8), 13 & (13,8), 13 & 0 & 0 & (13,8), 13 & 0 & 0 & (13,8), 13 & 0 & 0 \\
22 & (13,8), 13 & (13,8), 13 & 0 & 0 & (13,8), 13 & 0 & 0 & (13,8), 13 & 0 & 0 \\
23 & (13,8), 13 & (13,8), 13 & 0 & 0 & (13,8), 13 & 0 & 0 & (13,8), 13 & 0 & 0 \\
24 & (13,11), 13 & (13,11), 13 & 0 & 0 & (13,11), 13 & 0 & 0 & (13,11), 13 & 0 & 0 \\
25 & (13,11), 13 & (13,11), 13 & 0 & 0 & (13,11), 13 & 0 & 0 & (13,11), 13 & 0 & 0 \\
26 & (13,13), 13 & (13,13), 13 & 0 & 0 & (13,13), 13 & 0 & 0 & (13,13), 13 & 0 & 0 \\
27 & (14,13), 3,11 & (14,13), 3,11 & 0 & 0 & (14,13), 3,11 & 0 & 0 & (14,13), 3,11 & 0 & 0 \\
28 & (14,14), 11 & (14,14), 11 & 0 & 0 & (14,13), 3 & 0 & \textcolor{black}{\textbf{-1}} & (14,13), 3 & 0 & \textcolor{black}{\textbf{-1}} \\
29 & (16,13), 3,8,13 & (16,13), 3,8,13 & 0 & 0 & (16,13), 3,8,13 & 0 & 0 & (16,13), 3,8,13 & 0 & 0 \\
30 & (16,14), 3,13 & (16,14), 3,13 & 0 & 0 & (16,13), 8 & 0 & \textcolor{black}{\textbf{-1}} & (16,13), 8 & 0 & \textcolor{black}{\textbf{-1}} \\
31 & (17,14), 3 & (16,14), 3,13 & \textcolor{black}{\textbf{-1}} & 0 & (17,14), 3 & 0 & 0 & (16,13), 8 & \textcolor{black}{\textbf{-1}} & \textcolor{black}{\textbf{-1}} \\
32 & (19,13), 8,11 & (19,13), 8,11 & 0 & 0 & (19,13), 8,11 & 0 & 0 & (19,13), 8,11 & 0 & 0 \\
33 & (19,13), 8,11 & (19,13), 8,11 & 0 & 0 & (19,13), 8,11 & 0 & 0 & (19,13), 8,11 & 0 & 0 \\
34 & (21,13), 8,13 & (21,13), 8,13 & 0 & 0 & (21,13), 8,13 & 0 & 0 & (21,13), 8,13 & 0 & 0 \\
35 & (22,13), 11 & (22,13), 11 & 0 & 0 & (22,13), 11 & 0 & 0 & (22,13), 11 & 0 & 0 \\
36 & (22,14), 8 & (22,13), 11 & 0 & \textcolor{blue}{\textbf{-1}} & (22,13), 11 & 0 & \textcolor{black}{\textbf{-1}} & (22,13), 11 & 0 & \textcolor{black}{\textbf{-1}} \\
37 & (24,13), 11,13 & (24,13), 11,13 & 0 & 0 & (24,13), 11,13 & 0 & 0 & (24,13), 11,13 & 0 & 0 \\
38 & (24,14), 11 & (24,14), 11 & 0 & 0 & (24,13), 13 & 0 & \textcolor{black}{\textbf{-1}} & (24,13), 13 & 0 & \textcolor{black}{\textbf{-1}} \\
39 & (26,13), 13 & (26,13), 13 & 0 & 0 & (26,13), 13 & 0 & 0 & (26,13), 13 & 0 & 0 \\
40 & (26,14), 13 & (26,14), 13 & 0 & 0 & (26,14), 13 & 0 & 0 & (26,14), 13 & 0 & 0 \\
41 & (27,14), 13 & (26,14), 13 & \textcolor{black}{\textbf{-1}} & 0 & (27,14), 13 & 0 & 0 & (26,14), 13 & \textcolor{black}{\textbf{-1}} & 0 \\
42 & (26,16), 13 & (26,16), 13 & 0 & 0 & (26,16), 13 & 0 & 0 & (26,16), 13 & 0 & 0 \\
43 & (27,16), 13 & (26,16), 13 & \textcolor{black}{\textbf{-1}} & 0 & (27,16), 13 & 0 & 0 & (26,16), 13 & \textcolor{black}{\textbf{-1}} & 0 \\
44 & (27,17), 13 & (27,17), 13 & 0 & 0 & (27,16), 13 & 0 & \textcolor{black}{\textbf{-1}} & (26,16), 13 & \textcolor{black}{\textbf{-1}} & \textcolor{black}{\textbf{-1}} \\
45 & (26,19), 13 & (26,19), 13 & 0 & 0 & (26,19), 13 & 0 & 0 & (26,19), 13 & 0 & 0 \\
46 & (26,19), 13 & (26,19), 13 & 0 & 0 & (26,19), 13 & 0 & 0 & (26,19), 13 & 0 & 0 \\
47 & (26,21), 13 & (26,21), 13 & 0 & 0 & (26,21), 13 & 0 & 0 & (26,21), 13 & 0 & 0 \\
48 & (26,22), 13 & (26,22), 13 & 0 & 0 & (26,22), 13 & 0 & 0 & (26,22), 13 & 0 & 0 \\
49 & (27,22), 13 & (26,22), 13 & \textcolor{black}{\textbf{-1}} & 0 & (26,22), 13 & \textcolor{blue}{\textbf{-1}} & 0 & (26,22), 13 & \textcolor{black}{\textbf{-1}} & 0 \\
\hline
\end{tabular}
}
\end{center}
\end{table}

\begin{table}\caption{
The initial outcomes, together with PSPE moves for the ruleset $S = \{4,5,9\}$. 
%Deviation table for $S = \{4,5,9\}$. The $\delta_i$ represent the deviations from FvF for the respective settings and players.
}\label{tab:deviation459}
\begin{center}
{\small 
\begin{tabular}{|r|l|l|r|r|l|r|r|l|r|r|}
\hline
Heap & FvF  & FvA  & $\delta_\alpha$ & $\delta_\beta$ & AvF & $\delta_\alpha$ & $\delta_\beta$ & AvA & $\delta_\alpha$ & $\delta_\beta$ \\
\hline
0 & (0,0) & (0,0) & 0 & 0 & (0,0) & 0 & 0 & (0,0) & 0 & 0 \\
1 & (0,0) & (0,0) & 0 & 0 & (0,0) & 0 & 0 & (0,0) & 0 & 0 \\
2 & (0,0) & (0,0) & 0 & 0 & (0,0) & 0 & 0 & (0,0) & 0 & 0 \\
3 & (0,0) & (0,0) & 0 & 0 & (0,0) & 0 & 0 & (0,0) & 0 & 0 \\
4 & (4,0), 4 & (4,0), 4 & 0 & 0 & (4,0), 4 & 0 & 0 & (4,0), 4 & 0 & 0 \\
5 & (5,0), 5 & (5,0), 5 & 0 & 0 & (5,0), 5 & 0 & 0 & (5,0), 5 & 0 & 0 \\
6 & (5,0), 5 & (5,0), 5 & 0 & 0 & (5,0), 5 & 0 & 0 & (5,0), 5 & 0 & 0 \\
7 & (5,0), 5 & (5,0), 5 & 0 & 0 & (5,0), 5 & 0 & 0 & (5,0), 5 & 0 & 0 \\
8 & (5,0), 5 & (5,0), 5 & 0 & 0 & (5,0), 5 & 0 & 0 & (5,0), 5 & 0 & 0 \\
9 & (9,0), 9 & (9,0), 9 & 0 & 0 & (9,0), 9 & 0 & 0 & (9,0), 9 & 0 & 0 \\
10 & (9,0), 9 & (9,0), 9 & 0 & 0 & (9,0), 9 & 0 & 0 & (9,0), 9 & 0 & 0 \\
11 & (9,0), 9 & (9,0), 9 & 0 & 0 & (9,0), 9 & 0 & 0 & (9,0), 9 & 0 & 0 \\
12 & (9,0), 9 & (9,0), 9 & 0 & 0 & (9,0), 9 & 0 & 0 & (9,0), 9 & 0 & 0 \\
13 & (9,4), 9 & (9,4), 9 & 0 & 0 & (9,4), 9 & 0 & 0 & (9,4), 9 & 0 & 0 \\
14 & (9,5), 9 & (9,5), 9 & 0 & 0 & (9,5), 9 & 0 & 0 & (9,5), 9 & 0 & 0 \\
15 & (9,5), 9 & (9,5), 9 & 0 & 0 & (9,5), 9 & 0 & 0 & (9,5), 9 & 0 & 0 \\
16 & (9,5), 9 & (9,5), 9 & 0 & 0 & (9,5), 9 & 0 & 0 & (9,5), 9 & 0 & 0 \\
17 & (9,5), 9 & (9,5), 9 & 0 & 0 & (9,5), 9 & 0 & 0 & (9,5), 9 & 0 & 0 \\
18 & (9,9), 4,5,9 & (9,9), 4,5,9 & 0 & 0 & (9,9), 4,5,9 & 0 & 0 & (9,9), 4,5,9 & 0 & 0 \\
19 & (10,9), 5 & (10,9), 5 & 0 & 0 & (10,9), 5 & 0 & 0 & (10,9), 5 & 0 & 0 \\
20 & (10,9), 5 & (10,9), 5 & 0 & 0 & (10,9), 5 & 0 & 0 & (10,9), 5 & 0 & 0 \\
21 & (10,9), 5 & (10,9), 5 & 0 & 0 & (10,9), 5 & 0 & 0 & (10,9), 5 & 0 & 0 \\
22 & (13,9), 4,9 & (13,9), 4,9 & 0 & 0 & (13,9), 4,9 & 0 & 0 & (13,9), 4,9 & 0 & 0 \\
23 & (14,9), 5,9 & (14,9), 5,9 & 0 & 0 & (14,9), 5,9 & 0 & 0 & (14,9), 5,9 & 0 & 0 \\
24 & (14,10), 5 & (14,10), 5 & 0 & 0 & (14,9), 9 & 0 & \textcolor{black}{\textbf{-1}} & (14,9), 9 & 0 & \textcolor{black}{\textbf{-1}} \\
25 & (14,10), 5 & (14,10), 5 & 0 & 0 & (14,9), 9 & 0 & \textcolor{black}{\textbf{-1}} & (14,9), 9 & 0 & \textcolor{black}{\textbf{-1}} \\
26 & (14,10), 5 & (14,10), 5 & 0 & 0 & (14,9), 9 & 0 & \textcolor{black}{\textbf{-1}} & (14,9), 9 & 0 & \textcolor{black}{\textbf{-1}} \\
27 & (18,9), 9 & (18,9), 9 & 0 & 0 & (18,9), 9 & 0 & 0 & (18,9), 9 & 0 & 0 \\
28 & (18,10), 9 & (18,10), 9 & 0 & 0 & (18,10), 9 & 0 & 0 & (18,10), 9 & 0 & 0 \\
29 & (18,10), 9 & (18,10), 9 & 0 & 0 & (18,10), 9 & 0 & 0 & (18,10), 9 & 0 & 0 \\
30 & (18,10), 9 & (18,10), 9 & 0 & 0 & (18,10), 9 & 0 & 0 & (18,10), 9 & 0 & 0 \\
31 & (18,13), 9 & (18,13), 9 & 0 & 0 & (18,13), 9 & 0 & 0 & (18,13), 9 & 0 & 0 \\
32 & (18,14), 9 & (18,14), 9 & 0 & 0 & (18,14), 9 & 0 & 0 & (18,14), 9 & 0 & 0 \\
33 & (19,14), 9 & (18,14), 9 & \textcolor{black}{\textbf{-1}} & 0 & (19,14), 9 & 0 & 0 & (18,14), 9 & \textcolor{black}{\textbf{-1}} & 0 \\
34 & (19,14), 9 & (18,14), 9 & \textcolor{black}{\textbf{-1}} & 0 & (19,14), 9 & 0 & 0 & (18,14), 9 & \textcolor{black}{\textbf{-1}} & 0 \\
35 & (19,14), 9 & (18,14), 9 & \textcolor{black}{\textbf{-1}} & 0 & (19,14), 9 & 0 & 0 & (18,14), 9 & \textcolor{black}{\textbf{-1}} & 0 \\
36 & (18,18), 4,5,9 & (18,18), 4,5,9 & 0 & 0 & (18,18), 4,5,9 & 0 & 0 & (18,18), 4,5,9 & 0 & 0 \\
37 & (19,18), 5,9 & (19,18), 5,9 & 0 & 0 & (19,18), 5,9 & 0 & 0 & (19,18), 5,9 & 0 & 0 \\
38 & (19,19), 5 & (19,19), 5 & 0 & 0 & (19,18), 5,9 & 0 & \textcolor{black}{\textbf{-1}} & (19,18), 5,9 & 0 & \textcolor{black}{\textbf{-1}} \\
39 & (19,19), 5 & (19,19), 5 & 0 & 0 & (19,18), 5,9 & 0 & \textcolor{black}{\textbf{-1}} & (19,18), 5,9 & 0 & \textcolor{black}{\textbf{-1}} \\
40 & (22,18), 4,9 & (22,18), 4,9 & 0 & 0 & (22,18), 4,9 & 0 & 0 & (22,18), 4,9 & 0 & 0 \\
41 & (23,18), 5,9 & (23,18), 5,9 & 0 & 0 & (23,18), 5,9 & 0 & 0 & (23,18), 5,9 & 0 & 0 \\
42 & (23,19), 4,5,9 & (23,19), 5,9 & 0 & 0 & (23,18), 9 & 0 & \textcolor{black}{\textbf{-1}} & (23,18), 9 & 0 & \textcolor{black}{\textbf{-1}} \\
43 & (24,19), 5 & (23,19), 5,9 & \textcolor{black}{\textbf{-1}} & 0 & (24,19), 5 & 0 & 0 & (23,18), 9 & \textcolor{black}{\textbf{-1}} & \textcolor{black}{\textbf{-1}} \\
44 & (24,19), 5 & (23,19), 5,9 & \textcolor{black}{\textbf{-1}} & 0 & (24,19), 5 & 0 & 0 & (23,18), 9 & \textcolor{black}{\textbf{-1}} & \textcolor{black}{\textbf{-1}} \\
45 & (27,18), 9 & (27,18), 9 & 0 & 0 & (27,18), 9 & 0 & 0 & (27,18), 9 & 0 & 0 \\
46 & (27,19), 9 & (27,19), 9 & 0 & 0 & (27,19), 9 & 0 & 0 & (27,19), 9 & 0 & 0 \\
47 & (28,19), 9 & (27,19), 9 & \textcolor{black}{\textbf{-1}} & 0 & (28,19), 9 & 0 & 0 & (27,19), 9 & \textcolor{black}{\textbf{-1}} & 0 \\
48 & (28,19), 9 & (27,19), 9 & \textcolor{black}{\textbf{-1}} & 0 & (28,19), 9 & 0 & 0 & (27,19), 9 & \textcolor{black}{\textbf{-1}} & 0 \\
49 & (27,22), 9 & (27,22), 9 & 0 & 0 & (27,22), 9 & 0 & 0 & (27,22), 9 & 0 & 0 \\
\hline
\end{tabular}
}
\end{center}
\end{table}
\newpage
\begin{table}
\begin{center}
{\small 
\begin{tabular}{|r|l|l|r|r|l|r|r|l|r|r|}
\hline
Heap & FvF  & FvA  & $\delta \alpha$ & $\delta \beta$ & AvF & $\delta \alpha$ & $\delta \beta$ & AvA & $\delta \alpha$ & $\delta \beta$ \\
\hline
50 & (27,23), 9 & (27,23), 9 & 0 & 0 & (27,23), 9 & 0 & 0 & (27,23), 9 & 0 & 0 \\
51 & (28,23), 9 & (27,23), 9 & \textcolor{black}{\textbf{-1}} & 0 & (28,23), 9 & 0 & 0 & (27,23), 9 & \textcolor{black}{\textbf{-1}} & 0 \\
52 & (28,24), 9 & (28,24), 9 & 0 & 0 & (28,23), 9 & 0 & \textcolor{black}{\textbf{-1}} & (27,23), 9 & \textcolor{black}{\textbf{-1}} & \textcolor{black}{\textbf{-1}} \\
53 & (28,24), 9 & (28,24), 9 & 0 & 0 & (28,23), 9 & 0 & \textcolor{black}{\textbf{-1}} & (27,23), 9 & \textcolor{black}{\textbf{-1}} & \textcolor{black}{\textbf{-1}} \\
54 & (27,27), 4,5,9 & (27,27), 4,5,9 & 0 & 0 & (27,27), 4,5,9 & 0 & 0 & (27,27), 4,5,9 & 0 & 0 \\
55 & (28,27), 5,9 & (28,27), 5,9 & 0 & 0 & (28,27), 5,9 & 0 & 0 & (28,27), 5,9 & 0 & 0 \\
56 & (28,28), 4,5,9 & (28,28), 5,9 & 0 & 0 & (28,27), 5,9 & 0 & \textcolor{black}{\textbf{-1}} & (28,27), 5,9 & 0 & \textcolor{black}{\textbf{-1}} \\
57 & (29,28), 5 & (28,28), 5,9 & \textcolor{black}{\textbf{-1}} & 0 & (29,28), 5 & 0 & 0 & (28,27), 5,9 & \textcolor{black}{\textbf{-1}} & \textcolor{black}{\textbf{-1}} \\
58 & (31,27), 4,9 & (31,27), 4,9 & 0 & 0 & (31,27), 4,9 & 0 & 0 & (31,27), 4,9 & 0 & 0 \\
59 & (32,27), 5,9 & (32,27), 5,9 & 0 & 0 & (32,27), 5,9 & 0 & 0 & (32,27), 5,9 & 0 & 0 \\
60 & (32,28), 4,5,9 & (32,28), 5,9 & 0 & 0 & (32,27), 9 & 0 & \textcolor{black}{\textbf{-1}} & (32,27), 9 & 0 & \textcolor{black}{\textbf{-1}} \\
61 & (33,28), 5,9 & (32,29), 4 & \textcolor{black}{\textbf{-1}} & \textcolor{red}{\textbf{+1}} & (33,28), 5,9 & 0 & 0 & (32,27), 9 & \textcolor{black}{\textbf{-1}} & \textcolor{black}{\textbf{-1}} \\
62 & (33,29), 5 & (33,29), 5 & 0 & 0 & (33,28), 5,9 & 0 & \textcolor{black}{\textbf{-1}} & (32,27), 9 & \textcolor{black}{\textbf{-1}} & \textcolor{black}{\textbf{-2}} \\
63 & (36,27), 9 & (36,27), 9 & 0 & 0 & (36,27), 9 & 0 & 0 & (36,27), 9 & 0 & 0 \\
64 & (36,28), 9 & (36,28), 9 & 0 & 0 & (36,28), 9 & 0 & 0 & (36,28), 9 & 0 & 0 \\
65 & (37,28), 9 & (36,28), 9 & \textcolor{black}{\textbf{-1}} & 0 & (37,28), 9 & 0 & 0 & (36,28), 9 & \textcolor{black}{\textbf{-1}} & 0 \\
66 & (37,29), 9 & (37,29), 9 & 0 & 0 & (37,28), 9 & 0 & \textcolor{black}{\textbf{-1}} & (36,28), 9 & \textcolor{black}{\textbf{-1}} & \textcolor{black}{\textbf{-1}} \\
67 & (36,31), 9 & (36,31), 9 & 0 & 0 & (36,31), 9 & 0 & 0 & (36,31), 9 & 0 & 0 \\
68 & (36,32), 9 & (36,32), 9 & 0 & 0 & (36,32), 9 & 0 & 0 & (36,32), 9 & 0 & 0 \\
69 & (37,32), 9 & (36,32), 9 & \textcolor{black}{\textbf{-1}} & 0 & (37,32), 9 & 0 & 0 & (36,32), 9 & \textcolor{black}{\textbf{-1}} & 0 \\
70 & (37,33), 9 & (37,33), 9 & 0 & 0 & (38,32), 9 & \textcolor{red}{\textbf{+1}} & \textcolor{black}{\textbf{-1}} & (36,32), 9 & \textcolor{black}{\textbf{-1}} & \textcolor{black}{\textbf{-1}} \\
71 & (38,33), 9 & (37,33), 9 & \textcolor{black}{\textbf{-1}} & 0 & (38,33), 9 & 0 & 0 & (36,32), 9 & \textcolor{black}{\textbf{-2}} & \textcolor{black}{\textbf{-1}} \\
72 & (36,36), 4,5,9 & (36,36), 4,5,9 & 0 & 0 & (36,36), 4,5,9 & 0 & 0 & (36,36), 4,5,9 & 0 & 0 \\
73 & (37,36), 5,9 & (37,36), 5,9 & 0 & 0 & (37,36), 5,9 & 0 & 0 & (37,36), 5,9 & 0 & 0 \\
74 & (37,37), 4,5,9 & (37,37), 5,9 & 0 & 0 & (37,36), 5,9 & 0 & \textcolor{black}{\textbf{-1}} & (37,36), 5,9 & 0 & \textcolor{black}{\textbf{-1}} \\
75 & (38,37), 5,9 & (37,38), 4,5 & \textcolor{black}{\textbf{-1}} & \textcolor{red}{\textbf{+1}} & (38,37), 5,9 & 0 & 0 & (37,36), 5,9 & \textcolor{black}{\textbf{-1}} & \textcolor{black}{\textbf{-1}} \\
76 & (40,36), 4,9 & (40,36), 4,9 & 0 & 0 & (40,36), 4,9 & 0 & 0 & (40,36), 4,9 & 0 & 0 \\
77 & (41,36), 5,9 & (41,36), 5,9 & 0 & 0 & (41,36), 5,9 & 0 & 0 & (41,36), 5,9 & 0 & 0 \\
78 & (41,37), 4,5,9 & (41,37), 5,9 & 0 & 0 & (41,36), 9 & 0 & \textcolor{black}{\textbf{-1}} & (41,36), 9 & 0 & \textcolor{black}{\textbf{-1}} \\
79 & (42,37), 5,9 & (41,38), 4,9 & \textcolor{black}{\textbf{-1}} & \textcolor{red}{\textbf{+1}} & (42,37), 4,5,9 & 0 & 0 & (41,36), 9 & \textcolor{black}{\textbf{-1}} & \textcolor{black}{\textbf{-1}} \\
80 & (42,38), 5,9 & (42,38), 5,9 & 0 & 0 & (43,37), 5 & \textcolor{red}{\textbf{+1}} & \textcolor{black}{\textbf{-1}} & (41,36), 9 & \textcolor{black}{\textbf{-1}} & \textcolor{black}{\textbf{-2}} \\
81 & (45,36), 9 & (45,36), 9 & 0 & 0 & (45,36), 9 & 0 & 0 & (45,36), 9 & 0 & 0 \\
82 & (45,37), 9 & (45,37), 9 & 0 & 0 & (45,37), 9 & 0 & 0 & (45,37), 9 & 0 & 0 \\
83 & (46,37), 9 & (45,37), 9 & \textcolor{black}{\textbf{-1}} & 0 & (46,37), 9 & 0 & 0 & (45,37), 9 & \textcolor{black}{\textbf{-1}} & 0 \\
84 & (46,38), 9 & (46,38), 9 & 0 & 0 & (47,37), 9 & \textcolor{red}{\textbf{+1}} & \textcolor{black}{\textbf{-1}} & (45,37), 9 & \textcolor{black}{\textbf{-1}} & \textcolor{black}{\textbf{-1}} \\
85 & (45,40), 9 & (45,40), 9 & 0 & 0 & (45,40), 9 & 0 & 0 & (45,40), 9 & 0 & 0 \\
86 & (45,41), 9 & (45,41), 9 & 0 & 0 & (45,41), 9 & 0 & 0 & (45,41), 9 & 0 & 0 \\
87 & (46,41), 9 & (45,41), 9 & \textcolor{black}{\textbf{-1}} & 0 & (46,41), 9 & 0 & 0 & (45,41), 9 & \textcolor{black}{\textbf{-1}} & 0 \\
88 & (46,42), 9 & (46,42), 9 & 0 & 0 & (47,41), 9 & \textcolor{red}{\textbf{+1}} & \textcolor{black}{\textbf{-1}} & (45,41), 9 & \textcolor{black}{\textbf{-1}} & \textcolor{black}{\textbf{-1}} \\
89 & (47,42), 9 & (46,43), 9 & \textcolor{black}{\textbf{-1}} & \textcolor{red}{\textbf{+1}} & (47,42), 9 & 0 & 0 & (45,41), 9 & \textcolor{black}{\textbf{-2}} & \textcolor{black}{\textbf{-1}} \\
90 & (45,45), 4,5,9 & (45,45), 4,5,9 & 0 & 0 & (45,45), 4,5,9 & 0 & 0 & (45,45), 4,5,9 & 0 & 0 \\
91 & (46,45), 5,9 & (46,45), 5,9 & 0 & 0 & (46,45), 5,9 & 0 & 0 & (46,45), 5,9 & 0 & 0 \\
92 & (46,46), 4,5,9 & (46,46), 5,9 & 0 & 0 & (46,45), 5,9 & 0 & \textcolor{black}{\textbf{-1}} & (46,45), 5,9 & 0 & \textcolor{black}{\textbf{-1}} \\
93 & (47,46), 5,9 & (46,47), 4,5,9 & \textcolor{black}{\textbf{-1}} & \textcolor{red}{\textbf{+1}} & (47,46), 4,5,9 & 0 & 0 & (46,45), 5,9 & \textcolor{black}{\textbf{-1}} & \textcolor{black}{\textbf{-1}} \\
94 & (49,45), 4,9 & (49,45), 4,9 & 0 & 0 & (49,45), 4,9 & 0 & 0 & (49,45), 4,9 & 0 & 0 \\
95 & (50,45), 5,9 & (50,45), 5,9 & 0 & 0 & (50,45), 5,9 & 0 & 0 & (50,45), 5,9 & 0 & 0 \\
96 & (50,46), 4,5,9 & (50,46), 5,9 & 0 & 0 & (50,45), 9 & 0 & \textcolor{black}{\textbf{-1}} & (50,45), 9 & 0 & \textcolor{black}{\textbf{-1}} \\
97 & (51,46), 5,9 & (50,47), 4,9 & \textcolor{black}{\textbf{-1}} & \textcolor{red}{\textbf{+1}} & (51,46), 4,5,9 & 0 & 0 & (50,45), 9 & \textcolor{black}{\textbf{-1}} & \textcolor{black}{\textbf{-1}} \\
98 & (51,47), 5,9 & (51,47), 5,9 & 0 & 0 & (52,46), 5,9 & \textcolor{red}{\textbf{+1}} & \textcolor{black}{\textbf{-1}} & (50,45), 9 & \textcolor{black}{\textbf{-1}} & \textcolor{black}{\textbf{-2}} \\
99 & (54,45), 9 & (54,45), 9 & 0 & 0 & (54,45), 9 & 0 & 0 & (54,45), 9 & 0 & 0 \\
\hline
\end{tabular}
}
\end{center}
\end{table}

\end{document}